\pdfoutput=1
\documentclass[11pt,a4paper]{amsart}

\usepackage{comment}
\usepackage[english]{babel} %

\usepackage{amscd,amsmath,amssymb,amsthm} %

\usepackage[T1]{fontenc}
\usepackage{lmodern}

\usepackage[margin=1in]{geometry}

\usepackage{mathtools}
\usepackage{bm} %

\usepackage{slashed} %
\usepackage[babel=true]{microtype}
\usepackage[autostyle=true]{csquotes} %
\usepackage[dvipsnames]{xcolor}
\usepackage[biblatex=true]{embrac}

\usepackage[%
bookmarks=true,
colorlinks,
linkcolor=blue,
urlcolor=blue,
citecolor=blue,
plainpages=false,
pdfpagelabels,
final,
breaklinks=true,
pdfusetitle,
]{hyperref}
\usepackage[citestyle=numeric-comp,
            bibstyle=numeric-comp,
            backend=biber,
            url=false,
            doi=true,
            isbn=false,
            giveninits=true,
            maxbibnames=100,
            sorting=nyt]{biblatex}

\AtEveryBibitem{\clearfield{month}}
\AtEveryBibitem{\clearfield{day}}
\AtEveryBibitem{%
  \ifentrytype{thesis}
    {}
    {
      \ifentrytype{online}{}
      {
      \clearfield{url}%
      \clearfield{urldate}%
      }
    }%
}
            
\usepackage{tikz}
\usetikzlibrary{cd} %

\usepackage[shortlabels]{enumitem}
\newlist{myenumi}{enumerate}{1}
\setlist[myenumi,1]{label=\upshape(\roman*)}
\newlist{myenuma}{enumerate}{1}
\setlist[myenuma,1]{label=\upshape(\alph*)}
\usepackage{color}
\usepackage[textwidth=0.8in,disable]{todonotes}
\usepackage{thmtools}
\declaretheorem[name=Theorem, numberwithin=section]{thm}
\declaretheorem[name=Theorem, numbered=no]{thm*}
\declaretheorem[name=Lemma,numberlike=thm]{lem}
\declaretheorem[name=Lemma,numbered=no]{lem*}

\declaretheorem[name=Proposition,numberlike=thm]{prop}
\declaretheorem[name=Definition,numberlike=thm, style=definition]{defi}
\declaretheorem[name=Setup,numberlike=thm, style=definition]{setup}
\declaretheorem[name=Conjecture,numberlike=thm, style=remark]{conj}

\declaretheorem[name=Remark, numberlike=thm, style=remark]{rem}

\declaretheorem[name=Theorem]{thmx}

\numberwithin{equation}{section}
\allowdisplaybreaks[1]
\usepackage[noabbrev]{cleveref} %
\crefname{figure}{Figure}{Figures}
\crefname{table}{Table}{Tables}
\crefname{thm}{Theorem}{Theorems}
\crefname{thmx}{Theorem}{Theorems}
\crefname{lem}{Lemma}{Lemmas}
\crefname{defi}{Definition}{Definitions}
\crefname{setup}{Setup}{Setups}
\crefname{conj}{Conjecture}{Conjectures}
\crefname{quest}{Question}{Questions}
\crefname{cor}{Corollary}{Corollaries}
\crefname{corx}{Corollary}{Corollaries}
\crefname{prop}{Proposition}{Propositions}
\crefname{ex}{Example}{Examples}
\crefname{rem}{Remark}{Remarks}
\crefname{section}{Section}{Sections}
\crefname{subsection}{Subsection}{Subsections}
\crefname{chapter}{Chapter}{Chapters}
\crefname{appendix}{Appendix}{Appendices}
\crefdefaultlabelformat{#2\textup{#1}#3}
\creflabelformat{enumi}{(#2#1#3)}

\usepackage{crossreftools}
\usepackage{xparse}

\usepackage{ourmacros}

\title[DEC Shields]{Positive mass theorems for spin initial data sets with arbitrary ends and dominant energy shields}

\hypersetup{
  pdfauthor={Simone Cecchini, Martin Lesourd, Rudolf Zeidler}
}

\author{Simone Cecchini}
\thanks{S.~Cecchini was partially funded by the Deutsche Forschungsgemeinschaft (DFG, German Research Foundation) – Project number 441895604.}
\address[S.~Cecchini]{Department of Mathematics, Texas A\&M University, College Station, TX
77843, United States of America}
\email{\href{mailto:cecchini@tamu.edu}{cecchini@tamu.edu}}
\urladdr{\href{https://simonececchini.org}{www.simonececchini.org}}

\author{Martin Lesourd}
\thanks{M. Lesourd thanks the Gordon Betty Moore and John Templeton Foundations as well as the Harvard's Black Hole Initiative.}
\address[M.~Lesourd]{Black Hole Initiative, Harvard University, Cambridge, MA 02138, United States of America}
\email{\href{mailto:mlesourd@math.harvard.edu}{mlesourd@math.harvard.edu}}
\urladdr{\href{https://www.martinlesourd.com}{www.martinlesourd.com}}

\author{Rudolf Zeidler}
\thanks{R.~Zeidler was funded by the Deutsche Forschungsgemeinschaft (DFG, German Research Foundation) – Project numbers
523079177; %
427320536; %
390685587; %
338480246. %
}
\address[R.~Zeidler]{Mathematisches Institut, University of Münster, Germany}
\email{\href{mailto:math@rzeidler.eu}{math@rzeidler.eu}}
\urladdr{\href{https://www.rzeidler.eu}{www.rzeidler.eu}}

\begin{document}

\begin{abstract}
  We prove a positive mass theorem for spin initial data sets $(M,g,k)$ that contain an asymptotically flat end and a shield of dominant energy (a subset of $M$ on which the dominant energy scalar $\mu-|J|$ has a positive lower bound).
  In a similar vein, we show that for an asymptotically flat end $\mathcal{E}$ that violates the positive mass theorem (i.e. $\mathrm{E} < |\mathrm{P}|$), there exists a constant $R>0$, depending only on $\mathcal{E}$, such that any initial data set containing $\mathcal{E}$ must violate the hypotheses of Witten's proof of the positive mass theorem in an $R$-neighborhood of $\mathcal{E}$.
  This implies the positive mass theorem for spin initial data sets with arbitrary ends, and we also prove a rigidity statement.
  Our proofs are based on a modification of Witten's approach to the positive mass theorem involving an additional independent timelike direction in the spinor bundle.
\end{abstract}

\maketitle

\section{Introduction}
The positive mass theorem is a fundamental result in differential geometry and the geometry of scalar curvature that arose as a conjecture in general relativity, where it was formulated for \emph{asymptotically flat} manifolds $(M^n,g)$ (complete manifolds whose ends are asymptotic to Euclidean space $\mathbb{R}^n$). The theorem gives an inequality $\mass \geq 0$ where $\mass$ is a quantity computed at infinity in an asymptotically flat end, and a rigidity statement stating that if $\mass=0$ then $(M^n,g)$ is Euclidean space $\mathbb{R}^n$.
 
There are two main settings for the positive mass theorem, one that applies to Riemannian manifolds $(M,g)$ where the metric $g$ is assumed to have nonnegative scalar curvature $\scal\geq 0$ (the Riemannian case), and one which applies to \textit{initial data sets} $(M^n,g,k)$ satisfying the \textit{dominant energy condition} (DEC) $\mu\geq |J|$.
\begin{defi}\label{AFdata_intro}
An initial data set $(M^n,g,k)$ is a Riemannian manifold $(M^n,g)$ endowed with a symmetric two tensor $k$, and the dominant energy condition is
\begin{equation}
    \mu\geq |J|,
\end{equation}
where 
\begin{equation}\label{decineq}
\mu =\frac{1}{2}(\scal_g+(\tr_g(k))^2-|k|_g^2), \quad J=\operatorname{div} (k)- \D{}\tr_g(k).
\end{equation}
\end{defi}
The initial data set version of the positive mass theorem involves significant additional technicalities over the Riemannian case, and the appropriate rigidity statement for initial data sets also differs in that it involves the existence of an embedding into Minkowski spacetime $(\mathbb{R}^{1, n},\eta)$.
 
In $1988$, Schoen--Yau \cite{SY88} conjectured, in the Riemannian case, that the positive mass theorem holds for manifolds that need only have \textit{one} asymptotically flat end (i.e., there may be other complete ends but nothing about them is assumed other than the curvature condition). A number of more recent works on the subject have now settled the following \cite{Lesourd-Unger-Yau:Positive_mass_ends,LLU22a,LLU22b,Zhu22,BartnikChrusciel,CZ2021-PositiveMass}. 
\begin{thm}[Positive mass theorem for complete manifolds with arbitrary ends]\label{pmt.arbitrary.ends}
Let $(M^{n \geq 3},g)$ be a connected complete Riemannian manifold that contains a distinguished asymptotically flat end $\mathcal{E}$ and nonnegative scalar curvature $\scal \geq 0$. Assume either that $n\leq 7$ or that $M^n$ is spin. Then the ADM mass of $\mathcal{E}$ satisfies $\mass(\mathcal{E})\geq 0$, with equality if and only if $(M^n,g)$ is Euclidean space.
\end{thm}
\begin{rem}
In dimensions $3\leq n\leq 7$, the nonnegativity `$\mass(\mathcal{E})\geq 0$' in Theorem \ref{pmt.arbitrary.ends} was first proved by Lesourd--Unger--Yau \cite{Lesourd-Unger-Yau:Positive_mass_ends}. Later, J.~Zhu \cite{Z22} gave a different proof that also included the rigidity statement. Yet another proof of nonnegativity and rigidity was given by Lee--Lesourd--Unger \cite{LLU22a}, building on their own earlier work \cite{LLU21}. \\ \indent 
Under a spin assumption, the nonnegativity in Theorem \ref{pmt.arbitrary.ends} was proved by Bartnik-Chru\'{s}ciel~\cite{BC02}; in fact, their argument is more general and applies to initial data sets satisfying the DEC. More recently, in the Riemannian case, a different proof of both nonnegativity and rigidity was given by Cecchini\nobreakdash--Zeidler~\cite{CZ2021-PositiveMass}, whose methods are also able to cover cases of manifolds that need not be complete.
Subsequently, these techniques have been extended to the case of asymptotically hyperbolic ends by Chai--Wan~\cite{CW22}.
\end{rem}

In this paper, by building on \cite{CZ2021-PositiveMass}, we give a different proof of the inequality `$\mass\geq 0$' for spin initial data sets satisfying the DEC, and we prove the appropriate rigidity statement in all dimensions for spin initial data sets.
Our proof of the inequality differs significantly from that of Bartnik--Chru\'{s}ciel \cite{BC02}, and our method yields a new \textit{quantitative shielding theorem} for the DEC that is able to deal with incomplete manifolds. The main difficulty we have to overcome is that, whilst the method in \cite{CZ2021-PositiveMass} allows for incomplete manifolds and arbitrary ends, we now face the fact that $g$ is nonlinearly coupled to $k$. To deal with this, we introduce another independent `time' direction in the spin bundle and we consider a modified connection. As a result, we find that the associated Dirac operator satisfies a Schrödinger--Lichnerowicz formula which allows us to use some of the ideas in \cite{CZ2021-PositiveMass} that enabled dealing with incomplete manifolds and arbitrary ends. We expect that similar tricks may be useful in other contexts. 

For initial data sets, the definition of asymptotic flatness involves decay assumptions on both $g$ and $k$ (see~\cref{AFdata}), and moreover there are two asymptotic quantities of interest, the ADM energy $\admEnergy$ and ADM linear momentum $\admMomentum$ (see \cite{Lee2019-GeometricRelativity} for definitions and context). In this setting, the positive mass theorem is as follows. 
\begin{thm}[Positive mass theorem for initial data sets]\label{IDS.pmt}
Let $(M^{n \geq 3},g,k)$ be an asymptotically flat initial data set that satisfies the dominant energy condition $\mu-|J|\geq 0$. If either $n\leq 7$ or $M^n$ is spin, then $\admEnergy_{\mathcal{E}}\geq |\admMomentum_{\mathcal{E}}|$ for each asymptotically flat end \(\mathcal{E} \subset M\).% 
\end{thm}

\begin{rem}
Theorem \ref{IDS.pmt} was first proved for $n=3$ by Schoen--Yau \cite{Schoen-Yau81:General_Asymptotics}. Later, Witten \cite{Witten:positive_mass} gave a spinor argument for $n=3$ that was later made rigorous by Parker--Taubes \cite{Parker-Taubes:Wittens_Proof}, and was shown to hold for all $n$ (under a spin assumption). More recently, the theorem was proved for manifolds in the low dimensional range $3\leq n\leq 7$ by Eichmair\nobreakdash--Huang\nobreakdash--Lee\nobreakdash--Schoen~\cite{EHLS}. 
\end{rem}

The rigidity associated with Theorem \ref{IDS.pmt} is more subtle.  
\begin{thm}[Rigidity of positive mass theorem for initial data sets]\label{IDS.rigidity}
Let $(M^{n \geq 3},g,k)$ be a connected asymptotically flat initial data set satisfying $\mu-|J|\geq 0$. Then the following holds:
\begin{enumerate}
\item[\textup{(1)}] for $3\leq n\leq 7$, if $\admEnergy=0$, then $(M^n,g,k)$ embeds into Minkowski spacetime as a spacelike hypersurface,
\item[\textup{(2)}] 
for any $n$, if $(g,k)$ are assumed to satisfy certain extra decay assumptions on a given end and $\admEnergy=|\admMomentum|$ on that end, then in fact $\admEnergy=0$.
\end{enumerate}
\end{thm}

\begin{rem}
The proof of (1) was given by Eichmair \cite{E13} building on the $n=3$ case that had been done by Schoen--Yau \cite{SchoenYau:HypersurfaceMethod}. The proof of (1) does not require any extra decay assumptions on $g$ and $k$.\footnote{Eichmair used an extra decay condition on $\tr_g(k)$ for $n=3$, but this extra condition on $\tr_{g}(k)$ when $n=3$ can be lifted \cite{HKK22}.}
More recently, we note that Huang--Lee \cite{Huang-Lee.3} have given an elegant argument for (1) which uses slightly stronger assumptions but which holds for all $n$.  
\end{rem}

\begin{rem}\label{rem:Huang-Lee} 
Under a spin assumption, (2) was proved by Beig--Chru\'{s}ciel \cite{BC96} for $n=3$, and by Chru\'{s}ciel--Maerten \cite{CM05} for all $n$. 
Later, using ideas from \cite{BC96,CM05}, Huang--Lee \cite{Huang-Lee.1} gave an elegant proof of (2) that did not involve a spin assumption.
In another paper, Huang--Lee \cite{Huang-Lee.2} constructed $9$-dimensional asymptotically flat initial data sets satisfying $\mu-|J|\geq0$, $\admEnergy=|\admMomentum|$, $\admEnergy \neq 0$ but which do not satisfy the `extra decay' assumption, which shows that there is some form of extra decay assumption is necessary for (2). It is not known to us whether $9$ is a critical dimension for constructing such examples. 
\end{rem}
In view of Theorems \ref{pmt.arbitrary.ends} and \ref{IDS.pmt}, it is natural to expect the following.  
\begin{conj}[Nonnegativity for initial data sets with arbitrary ends]\label{dec.arbitrary.ends}
Let $(M^{n\geq 3},g,k)$ be a complete initial data set that contains at least one asymptotically flat end $\mathcal{E}$ and that satisfies $\mu-|J|\geq 0$. Then $\admEnergy_\mathcal{E}\geq |\admMomentum_\mathcal{E}|$. 
\end{conj}

\begin{rem}
If $M^n$ is spin, this was shown by Bartnik--Chru\'{s}ciel \cite{BartnikChrusciel}. 
\end{rem}
In view of the rigidity in Theorems \ref{pmt.arbitrary.ends} and \ref{IDS.rigidity}, the following is also natural.
\begin{conj}[Rigidity for initial data sets with arbitrary ends]\label{rigidityconj}
Let $(M^{n \geq 3},g,k)$ be a connected complete initial data set that contains at least one asymptotically flat end $\mathcal{E}$ and that satisfies $\mu-|J|\geq 0$. Then the following holds: 
\begin{enumerate}
\item[(1)] if $\admEnergy_\mathcal{E}=0$ then $(M^n,g,k)$ embeds as a spacelike hypersurface in Minkowski space,
\item[(2)] if $(g,k)$ on $\mathcal{E}$ has an extra decay and $\admEnergy_\mathcal{E}=|\admMomentum_\mathcal{E}|$, then $\admEnergy_\mathcal{E}=0$.
\end{enumerate}
\end{conj}

In view of the so-called \textit{quantitative shielding theorem} in \cite{Lesourd-Unger-Yau:Positive_mass_ends,CZ2021-PositiveMass,LLU22a}, it is also natural to pose the following. 
\begin{conj}[Positive mass theorem for initial data sets with shields]\label{dec.shield.conj}
Let $(M^n,g,k)$ be an initial data set that contains at least one asymptotically flat end $\mathcal{E}$ and a dominant energy shield $U_0 \supset U_1 \supset U_2$ as in Definition \ref{dec.shield} such that \(\mathcal{E} \subset U_2\) with \(\overline{U_0 \setminus \mathcal{E}}\) compact. Then $\admEnergy_\mathcal{E}>|\admMomentum_\mathcal{E}|$. 
\end{conj}

\begin{defi}[{compare \cref{fig:dec.shield}}]\label{dec.shield}
Let $(M^n,g,k)$ be an initial data set, not assumed to be complete.
A \emph{dominant energy shield} is a nested collection of connected non-empty open subsets $U_0 \supset U_1 \supset U_2$ such that $U_0 \supset \overline U_1$, $U_1\supset \overline U_2$, the closure of $U_0$ in $(M,g)$ is a complete manifold with compact boundary,
and we have the following:
\begin{enumerate}
    \item $\mu-|J|\ge 0$ on $U_0$,
    \item $\mu-|J| \ge \sigma n(n-1)$ on $U_1\setminus U_2$ for some constant $\sigma>0$,
    \item the mean curvature $\mean_{\partial\bar U_0}$ on $\partial\bar U_0$ and the symmetric two tensor $k$ satisfy%
    \[\mean_{\partial\bar U_0} - \tfrac{1}{n-1}\left\lvert k(\nu, \blank)|_{\T \partial \bar U_0}\right\rvert>-\Psi(d,l),\]
    where \(\nu\) is the unit normal field along \(\partial U_0\) pointing towards \(U_0\) and $\Psi(d,l)$ is the constant defined as
    \[
    \Psi(d,l) \coloneqq \begin{cases}  
    \frac{2}{n} \frac{\lambda(d)}{1-l \lambda(d)} & \text{if \(d < \frac{\pi}{\sqrt{\sigma} n}\) and \(l < \frac{1}{\lambda(d)}\),} \\
    \infty & \text{otherwise,}
    \end{cases} 
    \]
    where \(d \coloneqq \dist_g(\partial U_2, \partial U_1)\), \(l \coloneqq \dist_g(\partial U_1, \partial U_0)\), and
    \[
    \lambda(d) \coloneqq \frac{\sqrt{\sigma} n}{2} \tan\left(\frac{ \sqrt{\sigma} n d}{2}\right).
    \]
    In case any of the boundaries are empty (this can only happen if \(U_0\) itself is complete), then the involved distances are taken to be \(+ \infty\).%
\end{enumerate}
\end{defi}

Note that our convention for the mean curvature of hypersurfaces inside \(M\) is the average trace of the second fundamental form, that is, \(\mean = \frac{1}{(n-1)}\tr \mathrm{II}\).
The sign convention is such that the \((n-1)\)-sphere viewed as the boundary of the \(n\)-ball has mean curvature \(+1\).%
\begin{figure}[ht]
  \includegraphics{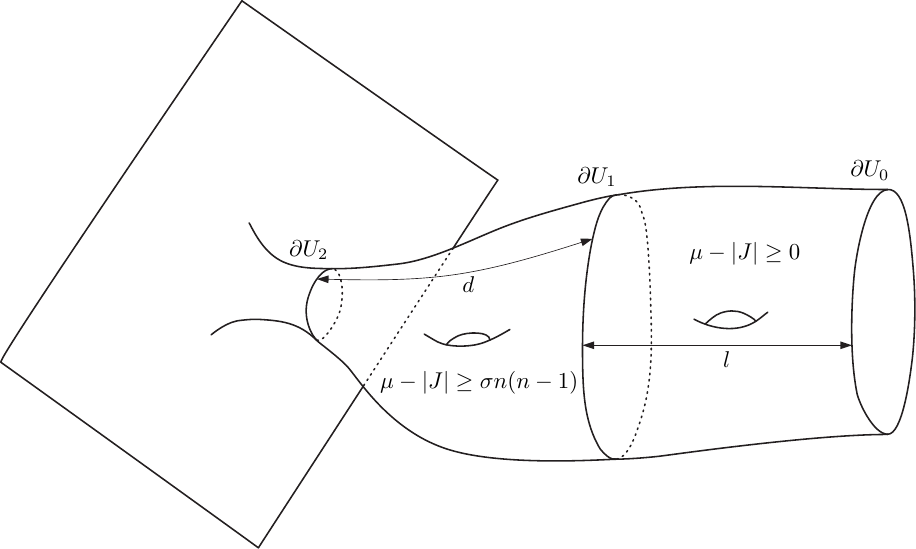}
  \caption{Dominant energy shield}
  \label{fig:dec.shield}
\end{figure}

\begin{rem}
 Observe that in (3) of Definition \ref{dec.shield}, the assumption on $k$ involves $|k(\nu,-)|_{\T\partial X}$, and not the trace $\textnormal{tr}_{\partial X}k$ which appears in the formula for the null mean curvature $\theta^+= (n-1)\mean + \tr_{\partial X}k$. 
 Naively, one might have expected (3) to involve $\theta^+$, since $\theta^+$ is the natural generalization of $\mean$.
 We impose this somewhat mysterious boundary condition because the term $|k(\nu,-)|_{\T\partial X}$ occurs as a natural bound for the boundary term that appears in our spinorial approach to the positive mass theorem with shields, see \cref{rem:chi1boundary}.
 This may be a technical artifact of our method.
However, it turns out that $|k(\nu,-)|_{\T\partial X}$ also figures in other treatments of initial data sets with boundary \cites[Definition~2.5]{almaraz:spacetimePMT_noncptbd}[Definition~2.3]{TinYau}, and, moreover, there is an indication that this term is to be expected from the perspective of the boundary Hamiltonian in general relativity (see the discussion in \cite{TinYau}).
\end{rem}

Here, using the aforementioned modification of the Callias operator method in \cite{Cecchini-Zeidler:ScalarMean}, we prove Conjecture \ref{dec.shield.conj} under a spin and compactness assumption.

\begin{thmx}\label{dec.shield.pmt}
Let $(M^{n\geq 3},g,k)$ be an initial data set that contains an asymptotically flat end $\mathcal{E}$ and a dominant energy shield as in \cref{dec.shield} such that \(\mathcal{E} \subset U_2\).
Assume that $ U_0$ is spin and that $\overline{U_0 \setminus \mathcal{E}}$ is compact. Then $\admEnergy_\mathcal{E}> |\admMomentum_\mathcal{E}|$. 
\end{thmx}
The proof of Theorem \ref{dec.shield.pmt} also leads to the following. 
\begin{thmx}\label{thm:spacetime_pmt_localized}
  Let \((\mathcal{E},g,k)\) be an asymptotically flat initial data end of dimension \(n \geq 3\) such that \(\admEnergy_{\mathcal{E}} < \lvert\admMomentum_{\mathcal{E}}\rvert\).
  Then there exists a constant \(R = R(\mathcal{E},g,k)\) such that the following holds: If \((M,g,k)\) is an \(n\)-dimensional initial data set \parensup{without boundary} that contains \((\mathcal{E}, g, k)\) as an open subset and \(\Nbh = \Nbh_R(\mathcal{E}) \subseteq M\) denotes the open neighborhood of radius \(R\) around \(\mathcal{E}\) in \(M\), then at least one of the following conditions must be violated:
  \begin{myenuma}
      \item \(\overline{\Nbh}\) is \parensup{metrically} complete, \label{item:complete}
      \item \(\mu-|J|\geq 0\) on \(\mathcal{U}\), \label{item:DEC}
      \item \(\Nbh\) is spin. \label{item:spin}
  \end{myenuma}
  \end{thmx}

Theorem \ref{thm:spacetime_pmt_localized} yields a new proof that \cref{dec.arbitrary.ends} holds in the spin setting:
  
\begin{thmx}\label{neater.than.bartnikchrusciel}
  Let $(M^{n \geq 3},g,k)$ be a complete initial data set that satisfies $\mu-|J|\geq 0$ and contains at least one distinguished asymptotically flat end $\mathcal{E}$. Assume that $M^n$ is spin. Then $\admEnergy_{\mathcal{E}}\geq \lvert \admMomentum_{\mathcal{E}}\rvert$.
\end{thmx}

It is worth noting that \cref{dec.shield.pmt,thm:spacetime_pmt_localized,neater.than.bartnikchrusciel} can accommodate lower regularity. As in the recent work of Lee--Lesourd--Unger \cite{LLU22b}, we can allow for the presence of corners (reminiscent of Miao~\cite{M02} and Shi--Tam \cite{ST02}), and we can also allow for distributional curvature as in Lee--LeFloch \cite{LeeLeFloch15}. The required modifications are fairly standard so we do not pursue this here.

Finally, we show the following rigidity statement. 
\begin{thmx}\label{rigidity.arb.ends}
Let $(M^{n \geq 3},g,k)$ be a connected complete spin initial data set that satisfies $\mu-|J|\geq 0$ and that contains at least one distinguished asymptotically flat end $\mathcal{E}$. Then the following holds: 
\begin{myenuma}
    \item if $\admEnergy_{\mathcal{E}}=0$, then $(M^n,g,k)$ embeds as a spacelike hypersurface in Minkowski spacetime,\label{rigidity.arb.end:item1}
    \item if $(g,k)$ on $\mathcal{E}$ has extra decay as in \cref{extra.decay} and $\admEnergy_{\mathcal{E}}=|\admMomentum_{\mathcal{E}}|$, then $\admEnergy_{\mathcal{E}}=0$.\label{rigidity.arb.end:item2}
\end{myenuma}
\end{thmx}

The paper is organized as follows. In \cref{defs} we introduce the basic definitions and notation. In Section \ref{sec:diracwitten_potential} we introduce the relevant spin bundle and compute some of the key formulae that underlie the remaining arguments. In Section \ref{mainfirstproofs} we prove Theorems \ref{dec.shield.pmt} and \ref{thm:spacetime_pmt_localized}. Finally in Section \ref{proof.of.rigidity} we prove the rigidity Theorem \ref{rigidity.arb.ends}. Useful weighted Poincar\'{e} inequalities for manifolds with boundary are included in Appendix \ref{weightedpoincare}. 

\subsection*{Acknowledgements} The authors thank Carla Cederbaum, Demetre Kazaras, Dan Lee, and Ryan Unger for insightful discussions.
They also thank the referee for carefully reading the manuscript and for many useful suggestions.

\section{Notations and Definitions}\label{defs}
Let a nonnegative integer $k$ and a constant $\tau\in\R$ be fixed.
Let $\Disk_d(0)\subset\R^n$ be the closed disc of radius \(d\) around the origin.
For a function $f$ on $\R^n\setminus\Disk_d(0)$, we write $f=\bigO_k(|x|^{-\tau})$ if the function $|x|^{\tau+\left|\beta\right|}|\partial^\beta f(x)|$ is bounded for each multi-index $0 \leq \left|\beta\right|\leq k$.
Moreover, for $\lambda\in (0,1)$, we write $f=\bigO_{k+\lambda}(|x|^{-\tau})$ if $f=\bigO_k(|x|^{-\tau})$ and the function
\[
    |x|^{\lambda+|\beta|+\tau}\frac{|\partial^\beta f(x)-\partial^\beta(y)|}{|x-y|^\lambda},\qquad x,y\in\R^n\setminus\Disk_d(0),\,0\leq|x-y|\leq|x|/2
\]
is bounded for every multi-index $|\beta|=k$.

\begin{defi}\label{AFdata}
Let $(X^n,g,k)$ be an initial data set.
We say that an open subset $\mathcal E\subset X$ is an \emph{asymptotically flat end} of order $\tau>\frac{n-2}{2}$ if $(g,k)$ is locally $\Ct^{2,\alpha}\times \Ct^{1,\alpha}$ for some $0<\alpha<1$, $\mu$ and $J$ are integrable in $\mathcal E$, and there exists a diffeomorphism \(\Phi \colon \mathcal{E} \xrightarrow{\cong} \R^n \setminus \Disk_{d}(0)\) for some \(d > 0\) such that, with respect to the induced coordinate chart $x=(x^1,\ldots,x^n)$, we have
\[
    g_{ij}-\delta_{ij}=\bigO_2(|x|^{-\tau})\ ,\qquad
    k_{ij}=\bigO_1(|x|^{-\tau-1})
\]
for all \(1 \leq i,j \leq n\).
We say that an initial data set $(X,g,k)$ is asymptotically flat if there exists a bounded set $K\subset X$ such that $X\setminus K$ is a non-empty disjoint union of finitely many asymptotically flat ends.%
\end{defi}

Following the convention of \cite[\S7.3.2]{Lee2019-GeometricRelativity}, we define the \emph{ADM energy-momentum} $(\admEnergy_\mathcal E,\admMomentum_\mathcal E)$ of an asymptotically flat end $\mathcal E$ as
\begin{align*}
    \admEnergy_\mathcal E\coloneqq&\lim_{r\to\infty}\frac{1}{2(n-1)\omega_{n-1}}\int_{S_r^{n-1}}\sum_{i,j=1}^n\left(\partial_ig_{ij}-\partial_jg_{ii}\right)\nu^j\,\D{}\bar S\\
    \left(\admMomentum_\mathcal E\right)_i\coloneqq&\lim_{r\to\infty}\frac{1}{(n-1)\omega_{n-1}}\int_{S_r^{n-1}}\sum_{j=1}^n\left(k_{ij}-\left(\tr_gk\right)g_{ij}\right)\nu^j\,\D{}\bar S    
\end{align*}
for $i=1,\ldots,n$, where $\omega_{n-1}$ is the volume of the $(n-1)$-dimensional unit sphere, $\Sphere_r^{n-1}$ is the $(n-1)$-sphere of radius $r$ with respect to the chosen asymptotically flat coordinates $x=(x^1,\ldots,x^n)$, $\nu^j=x^j/|x|$, and $\D{}\bar S$ is the volume element on $\Sphere_r^{n-1}$ with respect to the background Euclidean metric.
The numbers \((\admMomentum_{\mathcal{E}})_i\) are understood to represent the covector at infinity \(\admMomentum_{\mathcal{E}} = \sum_{i=1}^n (\admMomentum_{\mathcal{E}})_i\ \D{x^i}\) in \(\mathcal{E}\) in the coordinate chart \((x^1, \dotsc, x^n)\).
 
For Part \ref{rigidity.arb.end:item2} of \cref{rigidity.arb.ends}, we will use the following extra decay assumption. 
\begin{defi}\label{extra.decay}
Let $R>0$ and let $(g,k)$ be initial data on $\mathcal{E}\simeq\mathbb{R}^n\backslash \Disk_R(0)$.
We say that $(g,k)$ has \textit{extra decay} if
\begin{equation*}
    g_{ij}-\delta_{ij}=\bigO_{3+\lambda}(|x|^{-\alpha})\:\:,\:\: k_{ij}=\bigO_{2+\lambda}(|x|^{-1-\alpha})
\end{equation*}
\begin{equation*}
    \alpha>\  \begin{cases}
    \quad \frac{1}{2}&  \quad n=3\\
   \quad n-3  & \quad n\geq 4
\end{cases} \quad ,\quad \epsilon>0\quad,\quad 0<\lambda<1
\end{equation*}
\begin{equation*}\label{decay.on.mu.intro}
\mu=\bigO_{1+\lambda}(|x|^{-n-\epsilon})    
\end{equation*}
\end{defi}
Let $(X,g,k)$ be an $n$-dimensional asymptotically flat initial data set with compact boundary.
Let $\rho$ be a smooth positive function such that \(\rho=|x|\) outside a disk in each asymptotically flat end with respect to some asymptotically flat coordinate chart.
Moreover, we assume that \(\rho\) remains uniformly bounded away from \(0\) and \(\infty\) outside of the asymptotically flat ends.
Let $(E,\nabla)$ be a Hermitian vector bundle with metric connection on \(X\).
For $p\geq 1$ and $\delta\in\R$, we define the \emph{weighted Lebesgue space} $\Lp^p_{\delta}(X,E)$ as the space of sections $u\in \Lp^p_\loc(X,E)$ such that the weighted norm
\begin{equation*}
    \|u\|_{\Lp^p_\delta(X,E)}\coloneqq
    \begin{cases}
    \left(\int_X|u|^p\rho^{-\delta p-n} \dV\right)^{1/p} & p < \infty, \\
    \operatorname{ess\ sup}_{x \in X} \lvert u(x) \rvert \rho(x)^{-\delta} & p=\infty,
    \end{cases}
\end{equation*}
is finite.
For a nonnegative integer $k$, we define the \emph{weighted Sobolev space} $\SobolevW^{k,p}_\delta(X,E)$ as the space of sections $u\in \SobolevW^{k,p}_\loc(X,E)$ such that the weighted Sobolev norm
\begin{equation*}
    \|u\|_{\SobolevW^{k,p}_\delta(X,E)}
    \coloneqq\sum_{i=0}^k\|\nabla^iu\|_{\Lp^p_{\delta-i}(X,E)}
\end{equation*}
is finite.
When \(p = 2\), we use the usual notation \(\SobolevH^k_\delta(X,E) \coloneqq \SobolevW^{2,k}_\delta(X,E)\).

\section{Dirac--Witten operators with potential}\label{sec:diracwitten_potential}
We consider the following setup: Let \((X,g,k)\) be an initial data set possibly with boundary, and let $\psi$ be a potential, that is, a Lipschitz function \(\psi \colon X \to \R\).%
We consider the vector bundle \(\T^\ast X \oplus \R^2\), where we endow \(\R^2 = \langle y_0, y_1\rangle\) with the negative definite bundle metric \(-\D y_0^2 - \D y_1^2\) so that \(\T^\ast X \oplus \R^2\) is endowed with a bundle metric of signature \((n,2)\).
We assume furthermore that \(X\) is a spin manifold and let \(S \to X\) denote the spinor bundle associated to \(\T^\ast X \oplus \R^2\) with respect to some faithful representation \(\Cl_{n,2} \hookrightarrow \End(\Sigma)\). 
Let \(\epsilon_i = \clm(y_i) \colon S \to S\) be the Clifford action of \(y_i\).
We define modified connections on \(S\) by the following formulas
\begin{align*}
    \modnabla_\xi  &\coloneqq \nabla_\xi + \frac{1}{2} \clm(k(\xi, \blank)) \epsilon_0, \\ 
  \modnabla_\xi^\psi  &\coloneqq \nabla_\xi + \frac{1}{2} \clm(k(\xi, \blank)) \epsilon_0 - \frac{\psi}{n}\clm(\xi) \epsilon_1 = \modnabla_\xi  - \frac{\psi}{n}\clm(\xi) \epsilon_1, 
\end{align*}
and consider the associated Dirac operators
\begin{align*}
    \modDirac &\coloneqq \sum_{i=1}^n \clm(e^i) \modnabla_{e_i} = \Dirac - \frac{1}{2} \tr(k) \epsilon_0, \\
    \modDirac^\psi &\coloneqq \sum_{i=1}^n \clm(e^i) \modnabla_{e_i}^\psi = \Dirac - \frac{1}{2} \tr(k) \epsilon_0 + \psi \epsilon_1 = \modDirac + \psi \epsilon_1,
\end{align*}
where \(\Dirac = \sum_{i=1}^n \clm(e^i) \nabla_{e_i}\) is the unmodified Dirac operator.
The operator \(\modDirac\) is the usual Dirac--Witten~\cite{Witten:positive_mass} operator of the initial data set \((X,g,k)\) and \(\modDirac^\psi\) is the Dirac--Witten operator augmented with a Callias potential (the \(\psi \epsilon_1\) part).

\begin{rem}\label{rem:codim_2}
    The modified objects \(\modnabla\) and \(\modDirac\) behave as the restrictions of the corresponding data from a hypothetical spacetime in which the initial data set is embedded.
    In a similar vein, the connection \(\modnabla^\psi\) and the corresponding Dirac operator \(\modDirac^\psi\) may be thought of as arising from \(X\) being embedded as a spacelike submanifold of codimension \emph{two} with trivial normal bundle and second fundamental form \(k \oplus -\frac{2 \psi}{n} g\) inside a pseudo-Riemannian manifold of signature \((n,2)\).
    However, from the point of view of our applications only the first timelike normal dimension is of physical and geometric relevance, whereas the second additional one involving the potential \(\psi\) is only used as a technical tool.
\end{rem}

Direct calculation and the Schrödinger--Lichnerowicz formula for \(\mathcal{D}\) (see e.g.~\cite[Proposition~2.5]{GromovLawson:PSCDiracComplete}) yield the following formula:
\begin{prop}[Schrödinger--Lichnerowicz formula]\label{prop:Schroedinger-Lichnerowicz}
\[
(\modDirac^\psi)^2 = (\modnabla^\psi)^\ast \modnabla^\psi
+ \frac{1}{4}\left(\scal + \lvert\tr(k)\rvert^2 - \lvert k \rvert^2\right)
+ \frac{1}{2}\clm(\divg(k) - \D\ {\tr(k)})\epsilon_0
+ \frac{n-1}{n}(\psi^2 + \clm(\D{\psi}) \epsilon_1).
\]
\end{prop}

\begin{rem}
    The \emph{energy density} \(\mu = \frac{1}{2}(\scal + \tr(k)^2 - \lvert k \rvert^2)\) and \emph{momentum density} \(J = \divg(k) - \D\, {\tr(k)}\), see~\cite[Definition~7.16]{Lee2019-GeometricRelativity}, both appear in the formula above.
    It follows that, on sections compactly supported in the interior, we obtain the estimate
    \[
        (\modDirac^\psi)^2 \geq \frac{1}{2}(\mu - \lvert J \rvert) + \frac{n-1}{n}(\psi^2 - \lvert \D{\psi} \rvert).
    \]
    Note that, by definition, the \emph{dominant energy condition (DEC)} says that \(\mu - |J| \geq 0\).
\end{rem}

\begin{rem}
    This already includes the right normalization \({(n-1)}/{n}\) for the terms corresponding to the potential and no use of the Friedrich inequality and Penrose operator as in ~\cite{Cecchini-Zeidler:ScalarMean} is required.
    This is because we write the formula using the modified connection \(\modnabla^\psi\).
\end{rem}

Integrating \cref{prop:Schroedinger-Lichnerowicz} and partial integration yields:

\begin{prop} \label{prop:spectral_estimate}
For every compactly supported smooth section \(u\),%
\begin{align*}
    \int_X \lvert \modDirac^\psi u \rvert^2 \dV = &\int_X \left( \lvert \modnabla^\psi u \rvert^2 + \frac{1}{4}\left(\scal + \lvert\tr(k)\rvert^2 - \lvert k \rvert^2\right) \lvert u \rvert^2 \right) \dV \\
    &+ \int_X  \left( \frac{1}{2} \langle u, \clm(\divg(k) - \D\ {\tr(k)})\epsilon_0 u \rangle + \frac{n-1}{n}\langle u, (\psi^2 + \clm(\D{\psi}) \epsilon_1) u \rangle \right) \dV \\
    &+ \int_{\partial X} \left( \tfrac{n-1}{2} \mean \lvert u \rvert^2 - \langle u,  \mathcal{D}^\partial u \rangle \right) \dS \\ 
    &+ \int_{\partial X} \left( \frac{1}{2} \langle u, (\clm(k(\nu, \blank)) - \tr(k) \clm(\nu^\flat)) \epsilon_0 u \rangle + \frac{n-1}{n} \langle u, \psi \clm(\nu^\flat) \epsilon_1 u \rangle \right) \dS.
\end{align*}
Here \(\mean\) denotes the mean curvature of the boundary, \(\Dirac^\partial\) the canonical boundary Dirac operator, and \(\nu\) the interior normal field \embparen{compare~\textup{\cite[\S 2]{Cecchini-Zeidler:ScalarMean}} for the conventions}.
\end{prop}

For \(\chi_i \coloneqq \clm(\nu^\flat)\epsilon_i\), where \(i \in \{0,1\}\), we will use the notation \(\SobolevH^1_{\loc}(X,S; \chi_i)\) to denote the space of locally \(\SobolevH^1\)-sections on \(X\) which satisfy the boundary condition \(\chi_i u = u\) along \(\partial X\) (in the trace sense).
We will use this notation analogously for other function spaces other than \(\SobolevH^1_{\loc}\).%

\begin{rem}[Estimating the boundary term for the boundary condition \(\chi_0\)]
Suppose \(u \in \SobolevH^1_{\loc}(X,S; \chi_0)\).
    Then \(\langle u, \Dirac^\partial u \rangle = 0\) and \(\langle u, \clm(\nu^\flat) \epsilon_1 u \rangle = \langle u, \chi_1 u \rangle = 0\) along the boundary because \(\chi_0\) anti-commutes with \(\Dirac^\partial\) and \(\chi_1\).
    Thus the boundary term appearing in \cref{prop:spectral_estimate} is given by
    \begin{equation}\label{eq:boundary_spacetime_condition}
        \int_{\partial X} \frac{1}{2} \left( (n-1)\mean - \tr_{\partial X}(k) \right) \lvert u \rvert^2  \dS,
    \end{equation}
    compare~\cite[284]{Lee2019-GeometricRelativity}.
    Notably, the potential \(\psi\) completely drops out on the boundary, and thus cannot be used to dominate the other terms.
    However, \((n-1)\mean - \tr_{\partial X}(k) \geq 0\) corresponds to the boundary being weakly inner trapped in the sense of \cite[Defintion~7.28]{Lee2019-GeometricRelativity}.
\end{rem}

\begin{rem}[Estimating the boundary term for the boundary condition \(\chi_1\)]\label{rem:chi1boundary}
    Suppose \(u \in \SobolevH^1_{\loc}(X,S; \chi_1)\). Note that this corresponds to the correct boundary condition for the potential term.
    Then \(\langle u, \Dirac^\partial u \rangle = 0\) and \(\langle u, \clm(\nu^\flat) \epsilon_0 u \rangle = \langle u, \chi_0 u \rangle = 0\) along the boundary because \(\chi_1\) anti-commutes with \(\Dirac^\partial\) and \(\chi_0\). 
    Thus the boundary term appearing in \cref{prop:spectral_estimate} is given by
    \begin{equation}\label{eq:boundary_term_potential_condition}
        \int_{\partial X} \left( \tfrac{n-1}{2} \left(\mean + \tfrac{2}{n}\psi \right) \lvert u \rvert^2 + \frac{1}{2} \langle u, \clm(k(\nu, \blank)) \epsilon_0 u \rangle \right) \dS.
    \end{equation}
    Moreover, letting \(e_1, \dotsc, e_{n-1}\) be a local orthonormal frame of \(\T \partial X\), we obtain
    \begin{align*}
        \langle u, \clm(k(\nu, \blank)) \epsilon_0 u \rangle &= k(\nu, \nu) \underbrace{\langle u, \chi_0 u \rangle}_{=0} + \langle u, \clm(\sum_{i=1}^{n-1} k(\nu, e_i) e^i) \epsilon_0 u \rangle \\
        &\geq - \left(\sum_{i=1}^{n-1} k(\nu, e_i)^{2} \right)^{1/2} \lvert u \rvert^2= -\left\lvert k(\nu, \blank)|_{\T \partial X} \right\rvert \lvert u \rvert^2.
    \end{align*}
    Thus \eqref{eq:boundary_term_potential_condition} can be bounded below by%
    \[
        \int_{\partial X} \left( \tfrac{n-1}{2} \left(\mean + \tfrac{2}{n}\psi \right) - \tfrac{1}{2}\left\lvert k(\nu, \blank)|_{\T \partial X} \right\rvert \right) \lvert u \rvert^2 \dS.
        \]
\end{rem}

\section{Proof of \texorpdfstring{\cref{dec.shield.pmt,thm:spacetime_pmt_localized}}{Theorems \crtextractcref{reference}{dec.shield.pmt} and \crtextractcref{reference}{thm:spacetime_pmt_localized}}}\label{mainfirstproofs}
We use the setup and notation of \cref{sec:diracwitten_potential}.
In order to estimate  the quantity $\admEnergy_{\mathcal{E}} - \lvert \admMomentum_{\mathcal{E}} \rvert$, we make use of asymptotically constant spinors, which we will now define.
Let $(e_1,\ldots,e_n)$ be a tangent orthonormal frame in $\mathcal E$.
Note that it lifts to a section of the principal \(\Spin(n)\)-bundle over $\mathcal E$ and thus induces a trivialization of $S$ over $\mathcal E$.
We say that a spinor $u_0\in C^\infty(\mathcal E,S|_{\mathcal E})$ is constant with respect to the orthonormal frame $(e_1,\ldots,e_n)$ if it is constant with respect to this induced trivialization.
Moreover, we say that an orthonormal frame $(e_1,\ldots,e_n)$ is \emph{asymptotically constant} if there exist asymptotically flat coordinates $x^1, \ldots,  x^n$ such that $e_i=\sum_{j=1}^ne_i^j\frac{\partial}{\partial x^j}$ satisfies \(e^{j}_i - \delta_{ij} = \bigO_2(|x|^{-\tau})\), where \(\tau\) is the order of the end \(\mathcal{E}\).

\begin{defi}
We say that a spinor $u\in \SobolevH^1_{\loc}(X,S)$ is \emph{asymptotically constant in \(\mathcal{E}\)} if there exists a spinor $u_0\in C^\infty(\mathcal E,S|_\mathcal E)$ which is constant with respect to an asymptotically constant orthonormal frame and such that $u|_{\mathcal{E}} - u_0 \in \SobolevH^1_{-q}(\mathcal{E},S)$, where \(q\coloneqq (n-2)/2\).

In this case, the norm at infinity of \(u\) in \(\mathcal{E}\) is defined as \(|u|_{\mathcal{E}_\infty} \coloneqq |u_0| \in [0,\infty)\).
Moreover, for a co-vector $P = \sum_{i=1}^n P_i e^i$ at infinity in \(\mathcal{E}\), where \(P_i \in \mathbb{R}\), we define $\langle u,\cc(P)\epsilon_0 u\rangle_{\mathcal{E}_\infty}\coloneqq \langle u_0,\cc(P)\epsilon_0 u_0\rangle \in \R$.
\end{defi}

Note that an asymptotically flat orthonormal frame in $\mathcal E$ can be obtained by orthonormalizing the coordinate frame of an asymptotically flat coordinate chart.
Note also that, for an orthonormal frame $(e_1,\ldots,e_n)$ in $\mathcal E$ and a vector $P=(P_1,\ldots,P_n)$, we can choose a spinor $u_0$ which is constant which respect to this orthonormal frame and satisfies $|u_0|=1$ and $\langle u_0,\cc(P)\epsilon_0u_0\rangle=-|P|$.

\begin{lem}\label{lem:energy_momenum_estimate}
    Let \((X,g,k)\) be a complete connected asymptotically flat initial data set with compact boundary and let \(\psi \in \Cc^\infty(X, \R)\).
    Let  \(u \in \SobolevH^1_{\loc}(X,S; \chi_1)\) be a section that is asymptotically constant in each end \(\mathcal{E}\).
    Then
    \begin{multline*}
        \frac{1}{2}(n-1)\omega_{n-1}\sum_{\mathcal{E}} \left( \admEnergy_{\mathcal{E}} \lvert u \rvert_{\mathcal{E}_\infty}^2 + \langle u, \clm(\admMomentum_{\mathcal{E}}) \epsilon_0 u \rangle_{\mathcal{E}_\infty} \right) 
        + \|\modDirac^\psi u\|_{\Lp^2(X)}^2\\
        \geq \|\modnabla^\psi u\|_{\Lp^2(X)}^2 + \int_X \zeta^\psi  \lvert u \rvert^2 \dV + \int_{\partial X} \eta^\psi \lvert u \rvert^2\dS,        
    \end{multline*}
    where \((\admEnergy_\mathcal{E}, \admMomentum_\mathcal{E})\) is the ADM energy-momentum of the end \(\mathcal{E}\) and
    \begin{equation}\label{eq:theta_eta}
    \zeta^\psi = \frac{1}{2}(\mu - \lvert J \rvert) + \frac{n-1}{n}(\psi^2 - \lvert \D{\psi} \rvert) \qquad \eta^\psi = \frac{n-1}{2} \left( \mean + \tfrac{2}{n}\psi \right) - \frac{1}{2}\left\lvert k(\nu, \blank)|_{\T \partial X} \right\rvert.
    \end{equation}
\end{lem}
    \begin{proof}
        Combine \cref{prop:spectral_estimate,rem:chi1boundary} and the proof of \cite[Proposition~8.24]{Lee2019-GeometricRelativity}.
    \end{proof}

    We set \(q \coloneqq (n-2)/2 > 0\) and will use \(-q < 0\) as the decay rate for the weighted Sobolev spaces we consider in the following proposition and for the remainder of this section.
    Moreover, given a potential \(\psi \in \Cc^\infty(X,\R)\) on an initial data set \((X,g,k)\), we will use the expressions \(\zeta^\psi\) and \(\eta^\psi\) as introduced in \labelcref{eq:theta_eta}, and we will write \(\zeta^\psi = \zeta_+^\psi - \zeta_-^\psi\) with \(\zeta_\pm^\psi \geq 0\).
    
    \begin{prop}\label{prop:partial_injectivity}%
        Let \((X,g,k)\) be a complete connected asymptotically flat initial data set with compact boundary \(\partial X\).
        Fix a connected codimension zero submanifold  \(X_0 \subseteq X\) with compact boundary \(\partial X_0\) which contains at least one asymptotically flat end of \(X\).
        Then there exists a constant \(c = c(X_0,g,k) > 0\), depending only on the data on \(X_0\), such that the following holds:
        
        For any potential \(\psi \in \Cc^\infty(X,\R)\) with \(\supp(\zeta^\psi_-) \subseteq X_0\) and \(\eta^\psi \geq 0\) on \(\partial X\) such that the followings holds,
        \begin{myenumi}
            \item \(\| \zeta^\psi_- \|_{\Lp^\infty_{-2}(X_0)} \leq c/2\),
            \item \(\| \psi^2 / n^2 \|_{\Lp^\infty_{-2}(X_0)} \leq c/2\),
        \end{myenumi}
        we have that
        \[
          \modDirac^\psi \colon \SobolevH^1_{-q}(X,S;\chi_1) \to \Lp^2(X,S)  
        \]
        is an isomorphism and 
        \begin{equation}
            \|\modDirac^\psi u\|_{\Lp^2(X)}^2 \geq c \| u \|^2_{\Lp^2_{-q}(X_0)} \label{eq:partial_injectivity_estimate}
        \end{equation}
        for all \(u \in \SobolevH^1_{-q}(X,S; \chi_1)\).
    \end{prop}
    \begin{proof}
        By \cref{lem:weighted_poincare_modnabla,rem:apply_weighted_poincare}, there is a constant \(c_0 = c_0(X_0,g,k) > 0\) such that \(c_0 \|v\|_{\Lp^2_{-q}(X_0)}^2 \leq \| \modnabla v \|^2_{\Lp^2(X_0)} \) for all \(v \in \SobolevH_{-q}^1(X_0,S)\).
        We now set \(c = c_0 / 2\).
        Then for all \(u \in \SobolevH^1_{-q}(X,S;\chi_1)\) \cref{lem:energy_momenum_estimate} implies 
        \begin{align*}
            \left\| \modDirac^\psi u \right\|^2_{\Lp^2(X)} &\geq \left\| \modnabla^\psi u \right\|^2_{\Lp^2(X)} + \int_{X} \zeta^\psi \lvert u \rvert^2 \dV + \int_{\partial X} \underbrace{\eta^\psi \lvert u \rvert^2}_{\geq 0} \dS \\
            &\geq \left\| \modnabla^\psi u \right\|^2_{\Lp^2(X)} - \int_{X_0} \zeta^\psi_- \lvert u \rvert^2 \dV \\
            &\geq \left\| \modnabla^\psi u \right\|^2_{\Lp^2(X_0)} - \int_{X_0} \zeta^\psi_- \lvert u \rvert^2 \dV \\
            &\geq \left\| \modnabla u \right\|^2_{\Lp^2(X_0)} - \left\|\tfrac{\psi}{n} u\right\|_{\Lp^2(X_0)}^2 - \int_{X_0} \zeta^\psi_- \lvert u \rvert^2 \dV \\
            &\geq c_0 \| u \|^2_{\Lp^2_{-q}(X_0)} - \left\| \tfrac{\psi^2}{n^2} \right\|_{\Lp^\infty_{-2}(X_0)} \| u \|_{\Lp^2_{-q}(X_0)}^2 - \left\|\zeta^\psi_-\right\|_{\Lp^\infty_{-2}(X_0)} \|u \|_{\Lp^2_{-q}(X_0)}^2\\
            & \geq (c_0 - c/2 - c/2) \| u \|^2_{\Lp^2_{-q}(X_0)} = c \| u \|^2_{\Lp^2_{-q}(X_0)}. 
        \end{align*}
        So we have established \labelcref{eq:partial_injectivity_estimate}.
        A similar argument as in the proof of \cite[Theorem~2.12]{CZ2021-PositiveMass} shows that \(\modDirac^\psi\) is a Fredholm operator with $\dim\coker(\modDirac^\psi)\leq\dim\ker(\modDirac^\psi)$.
        Thus, to see that \(\modDirac^\psi\) is an isomorphism, it suffices to show that it has trivial kernel.
        Indeed, if \(u \in \SobolevH^1_{-q}(X,S;\chi_1)\) with \(\modDirac^\psi u = 0\), then \labelcref{eq:partial_injectivity_estimate} implies that \(\|u\|_{\Lp^2_{-q}(X_0)} = 0\). In particular, 
        \(u\) vanishes on \(X_0\).
        In this case, the first two lines in the previous estimate imply that \(\left\| \modnabla^\psi u \right\|_{\Lp^2(X)} = 0\) and hence \(\modnabla^\psi u = 0\) on all of \(X\).
        Since \(u\) vanishes on \(X_0\) and \(X\) is connected, this already implies that \(u = 0\) everywhere.%
    \end{proof}
    
    Note that, in general, a potential \(\psi\) satisfying all the hypotheses from \cref{prop:partial_injectivity} may not exist.
    In the proofs of our main results that follow one of the main tasks is to construct such an appropriate potential.%
    We now proceed with the proof of \cref{dec.shield.pmt}.
    
    \begin{proof}[Proof of \cref{dec.shield.pmt}]%
    Let $(M^n,g,k)$ be an initial data set of dimension \(n \geq 3\) containing an asymptotically flat end $\mathcal{E}$ and a dominant energy shield $\mathcal{E} \subset U_2\subset U_1\subset U_0$ as in \cref{dec.shield}.
    Moreover, assume that $U_0$ is spin and that $\overline{U_0 \setminus \mathcal{E}}$ is compact.
    In particular, $\mathcal E$ is the only end contained in $U_0$.
    Recall that $\bar U_0$ is a complete manifold with compact boundary.
    Let us first use Conditions (1)--(3) of \cref{dec.shield} to construct a potential $\psi\in \Ct^\infty_\cc(\bar U_0,\R)$ such that $\psi|_{U_2}=0$, $\zeta^\psi\geq0$ in $\bar U_0$, and $\eta^\psi>0$ in $\partial\bar U_0$, where $\zeta^\psi$ and $\eta^\psi$ are defined by \eqref{eq:theta_eta}.
    
    Using Condition (3) of \cref{dec.shield}, choose $d^\prime<d$ and $l^\prime<l$ such that $d^\prime<\frac{\pi}{\sqrt{\sigma} n}$, \(l^\prime < \frac{1}{\lambda(d^\prime)}\), and 
    \begin{equation}\label{eq:MeanApprox}
        \mean_{\partial\bar U_0} - \tfrac{1}{n-1}\left\lvert k(\nu, \blank)|_{\T \partial \bar U_0}\right\rvert>-\Psi(d^\prime,l^\prime)>-\infty.
    \end{equation}
    Let $V\subset U_1\setminus\overline U_2$ be a compact manifold with boundary $\partial V=\partial_-V\sqcup\partial_+V$, where $\partial_\pm V$ are separating hypersurfaces of $M$ satisfying the following properties: $\dist_g(\partial_+V,\partial_-V)>d^\prime$ and $\bar{U}_0=V_1\cup_{\partial_-V}V\cup_{\partial_+V}V_2$, where $V_1$ and $V_2$ are smooth manifolds with compact boundaries $\partial V_1=\partial_-V$ and $\partial V_2=\partial_+V\sqcup\partial\bar U_0$ respectively.
    Note that such manifold $V$ can be constructed using the regular points of a smooth function $\Ct^0$-close to the distance function from $\bar{U}_2$.
    Note also that $\mu-|J| \ge \sigma n(n-1)$ on $V$ by Condition (2) of \cref{dec.shield}.
    Since $d^\prime \sqrt\sigma n<\pi$ and since $\dist_g(\partial_+V,\partial_-V)>d^\prime$, using \cite[Lemma 3.2]{CZ2021-PositiveMass} with $\delta=d^\prime$ and $\eta=\frac{\sqrt\sigma n}{2}$, there exists a smooth function $\psi_1\colon V\to[0,\lambda(d^\prime)]$ such that $\psi_1=0$ in a neighborhood of $\partial_-V$, $\psi_1=\lambda(d^\prime)$ in a neighborhood of $\partial_+V$, and $\zeta^{\psi_1}\geq0$ in $V$.
    Let $W\subset \bar{U}_0\setminus\overline U_1$ be a compact manifold with boundary $\partial W=\partial_-W\sqcup\partial\bar U_0$, where $\partial_- W$ is a separating hypersurface of $M$ such that $\dist_g(\partial_-W,\partial\bar U_0)>l^\prime$ and $\bar{U}_0=W_1\cup_{\partial_-W}W$, where $W_1$ is a smooth manifold with compact boundary $\partial W_1=\partial_-W$.
    Note that $\mu-|J| \ge 0$ on $W$ by Condition (1) of \cref{dec.shield}.
    Since $\dist_g(\partial_-W,\partial\bar U_0)>l^\prime$ and \(l^\prime < \frac{1}{\lambda(d^\prime)}\), by \cite[Lemma 3.3]{CZ2021-PositiveMass} applied with $\delta=l^\prime$ and $\lambda=\lambda(d^\prime)$, there exists a smooth function $\psi_2\colon W\to[\lambda(d^\prime),\infty)$ such that $\psi_2=\lambda(d^\prime)$ in a neighborhood of $\partial_-W$, $\psi_2=\lambda(d^\prime)/\left(1-l^\prime\lambda(d^\prime)\right)$ in a neighborhood of $\partial\bar U_0$, and $\psi_2^2-\lvert \D{\psi} \rvert\geq0$ in $W$.
    Since $\mu-|J| \ge 0$ on $W$,
    \[
        \zeta^{\psi_2}=\frac{1}{2}(\mu - \lvert J \rvert) + \frac{n-1}{n}(\psi_2^2 - \lvert \D{\psi_2} \rvert)\ge\frac{n-1}{n}(\psi_2^2 - \lvert \D{\psi_2} \rvert)\ge0\qquad\textrm{in }W. 
    \]
    Finally, let $\psi\in \Ct^\infty_\cc(\bar U_0,\R)$ be the potential function defined by setting $\psi=0$ in $V_1$, $\psi=\psi_1$ in $V$, $\psi=\lambda(d^\prime)$ in $V_2\cap W_1$, $\psi=\psi_2$ in $W$.
    Note that $\psi|_{U_2}=0$ and $\zeta^\psi\geq0$ in $\bar U_0$.
    Moreover, \(\psi=\lambda(d^\prime)/\left(1-l^\prime\lambda(d^\prime)\right)=\frac{n}{2}\Psi(d^\prime,l^\prime)\) on $\partial\bar U_0$.
    Using \eqref{eq:MeanApprox},
    \[
    \eta^\psi = \frac{n-1}{n}\psi+\frac{n-1}{2}\left(\mean_{\partial\bar U_0} - \tfrac{1}{n-1}\left\lvert k(\nu, \blank)|_{\T \partial \bar U_0}\right\rvert\right)
    >\frac{n-1}{n}\left(\psi-\frac{n}{2}\Psi(d^\prime,l^\prime)\right)=0
    \]
    on $\partial\bar U_0$, showing that $\psi$ satisfies the desired properties.
    
    Let \(u_{00} \in \Ct^\infty(U_0,S)\) be a spinor with \(\supp(u_{00}) \subset U_2\) which is asymptotically constant in \(\mathcal{E}\) and such that
        \[
          \lvert u_{00} \rvert_{\mathcal{E}_\infty} = 1, \qquad 
          \langle u_{00}, \clm(\admMomentum_{\mathcal{E}}) \epsilon_0 u_{00} \rangle_{\mathcal{E}_\infty} = - \lvert\admMomentum_{\mathcal{E}}\rvert.
        \]
    Using \cref{prop:partial_injectivity} with \(X = U_0\) and \(X_0 = U_2\), pick a spinor \(v\in\SobolevH^1_{-q}(\bar U_0,S;\chi_1)\) such that
    \begin{equation}\label{eq:WittenSpinor}
        \modDirac^{\psi}v=-\modDirac(u_{00})=-\modDirac^{\psi}(u_{00}),
    \end{equation}
    where the last equality holds because \(\psi|_{U_2} = 0\) and \(\supp(u_{00}) \subset U_2\).%
    Then \(u \coloneqq u_{00}+v\) is an asymptotically constant spinor which is asymptotic to \(u_{00}\) in \(\mathcal{E}\) and satisfies \(\modDirac^{\psi}u=0\).
    Since $\zeta^\psi\geq0$ in $\bar U_0$, and $\eta^\psi>0$ in $\partial\bar U_0$, by \cref{lem:energy_momenum_estimate}
    \begin{equation*}
        \frac{1}{2}(n-1)\omega_{n-1} \left( \admEnergy_{\mathcal{E}} - \lvert \admMomentum_{\mathcal{E}} \rvert \right)
        \geq \|\modnabla^\psi u\|_{\Lp^2(\bar U_0)}^2 + \int_{\bar U_0} \zeta^\psi  \lvert u \rvert^2 \dV 
        +\int_{\partial\bar U_0} \eta^\psi \lvert u \rvert^2\dS>0
    \end{equation*}
    which concludes the proof.
    \end{proof}

    We now turn to the proof of \cref{thm:spacetime_pmt_localized}, the main steps of which are contained in the following lemmas. 
    We work in the following setup which is designed to prove \cref{thm:spacetime_pmt_localized} by contrapositive.
    
    \begin{setup}\label{setup:sweeping_out_M}
        Let \((M,g,k)\) be a connected initial data set and \(\mathcal{E} \subseteq M\) a fixed asymptotically flat end.
        \begin{myenumi}
            \item 
        Fix a connected codimension zero submanifold \(X_0 \subseteq M\) such that \((X_0, g)\) is complete and has a single asymptotically flat end coinciding with \(\mathcal{E}\) at infinity.
        We also fix a collar neighborhood \(K \cong \partial X_0 \times [-1,0]\) of \(\partial X_0\) inside \(X_0\) and a smooth cut-off function \(\chi \colon X_0 \to [0,1]\) such that \(\chi = 0\) on \(X_0 \setminus K\) and \(\chi = 1\) near \(\partial X_0\).\label{item:setup_X0}
        \item Now given a distance \(r > 0\), we assume that there exists another connected codimension zero submanifold \(X_0 \subset X_r \subseteq M\) such that \(\dist_g(\partial X_0, \partial X_r) > r\), \((X_r, g)\) is a complete spin manifold such that \(\mu - \lvert J \rvert \geq 0\) holds on \(X_r\), and \(X_r \setminus X_0\) is relatively compact.\label{item:setup_Xr}
        \end{myenumi}
    \end{setup}
    
    \begin{figure}[t]
        \includegraphics{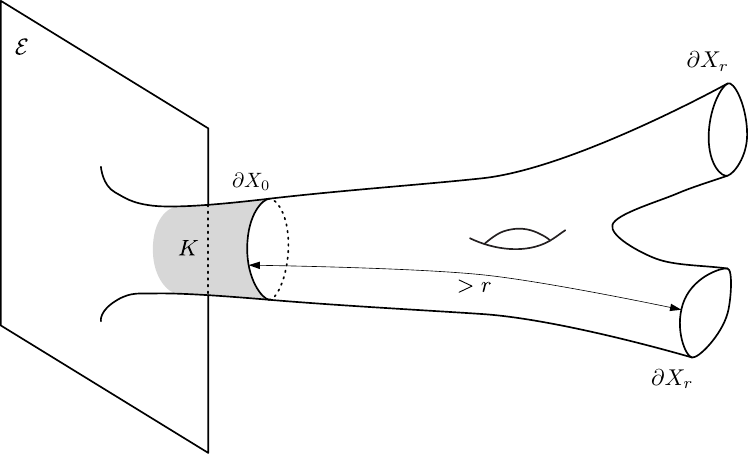}
        \caption{\cref{setup:sweeping_out_M}}
    \end{figure}
    
    \begin{lem}\label{lem:potential_exists}
        Given a distance \(r > 0\) and assuming we are in \cref{setup:sweeping_out_M}, we can find a potential \(\psi_r \in \Cc^\infty(X_r,[0,\infty))\) such that the following holds:
        \begin{alignat}{2}
            \psi_r &= \frac{\chi}{r} \quad &&\text{on \(X_0\)}, \label{eq:potential_on_end}\\
            \psi_r(x) &< \frac{4}{r}\quad && \text{if \(\dist_g(x, X_0)< \frac{r}{2} \),} \label{eq:potential_upper_bound}\\
            \psi_r^2 - \lvert \D{\psi_r}\rvert &\geq 0 \quad &&\text{on \(X_r \setminus X_0\)}, \label{eq:riccati_nonneg}\\
            \eta^{\psi_r} &\geq 0 \quad &&\text{on \(\partial X_r\)}, \label{eq:boundary_nonneg}
        \end{alignat}
        where \(\eta^{\psi_r}\) is defined as in \labelcref{eq:theta_eta}.
    \end{lem}
    \begin{proof}
        Since the function \(\psi_r\) is already determined on \(X_0\) by \labelcref{eq:potential_on_end}, we only need to extend it to \(X_r\) while satisfying \labelcref{eq:potential_upper_bound,eq:riccati_nonneg,eq:boundary_nonneg}.
        We can achieve this by taking a smooth and appropriately cut-off version of the function \(x \mapsto 1/(r-t)\) with \(t = \dist_g(\partial X_0, x)\) for \(x \in X_r \setminus X_0\).
        The details of this construction were essentially worked out in \cite[Lemma~3.3]{CZ2021-PositiveMass}.
    \end{proof}
    
    \begin{lem}\label{lem:isomorphism_in_setup}
        Given \cref{setup:sweeping_out_M}, there exist constants \(c = c(X_0,g,k)\) and \(r_0 = r_0(X_0,g,k,\chi) > 0\), depending only on the data on \(X_0\), such that if \(r > r_0\) and \(\psi_r \in \Cc^\infty(X_r, [0,\infty))\) satisfies \labelcref{eq:potential_on_end,eq:riccati_nonneg,eq:boundary_nonneg}, then the following holds:
        The operator \(\modDirac^{\psi_r} \colon \SobolevH^1_{-q}(X_r, S; \chi_1) \to \Lp^2(X_r, S)\) is an isomorphism and 
        \begin{equation}
            \forall v \in \SobolevH^1_{-q}(X_r, S; \chi_1) \colon \quad \| \modDirac^{\psi_r} v \|_{\Lp^2(X_r)}^2 \geq c \| v \|^2_{\Lp^2_{-q}(X_0)}. \label{eq:applied_partial_injectivity_bound}
        \end{equation}
    \end{lem}
    \begin{proof}
        Let \(c = c(X_0,g,k) > 0\) be the constant from \cref{prop:partial_injectivity}.
        Since \(\mu - \lvert J \rvert \geq 0\) on \(X_r\) and using \(\psi_r = 0\) on \(X_0 \setminus K\) together with \labelcref{eq:riccati_nonneg}, we obtain that
        \begin{equation}
          \zeta^{\psi_r} \coloneqq \frac{1}{2}(\mu - \lvert J \rvert) + \frac{n-1}{n}(\psi_r^2 - \lvert \D{\psi_r} \rvert) \geq \begin{cases}
            0 & \text{on \(X_r \setminus K\),} \\
            - \frac{1}{r} \lvert \D{\chi} \rvert & \text{on \(K\).}
          \end{cases}\label{eq:applied_zeta}
        \end{equation}
        Together with \labelcref{eq:potential_on_end} and since \(\chi\) is supported in the compact set \(K\), we can choose a constant \(r_0 = r_0(X_0, g, k, K,f ) > 0\) such that the conditions
        \begin{myenumi}
            \item \(\| \zeta^{\psi_r}_- \|_{\Lp^\infty_{-2}(X_0)} \leq \frac{1}{r} \|\rho^2 \D{f}\|_{\Lp^\infty(K)} \leq c/2\), \label{item:bound_zeta_minus}
            \item \(\| {\psi_r}^2 / n^2 \|_{\Lp^\infty_{-2}(X_0)} \leq \frac{1}{n^2 r^2} \| \rho^2 \|_{\Lp^\infty(K)} \leq c/2\), \label{item:bound_psi}
        \end{myenumi}
        hold for all \(r > r_0\).
        Thus \cref{prop:partial_injectivity} implies for all \(r > r_0\) that \(\modDirac^{\psi_r} \colon \SobolevH^1_{-q}(X_r, S; \chi_1) \to \Lp^2(X_r, S)\) is an isomorphism and that \labelcref{eq:applied_partial_injectivity_bound} holds.
    \end{proof}

    \begin{lem}
        Given \cref{setup:sweeping_out_M}, let \(c = c(X_0,g,k)\) and \(r_0 = r_0(X_0,g,k,\chi) > 0\) be the constants obtained from \cref{lem:isomorphism_in_setup} and let \(\psi_r \in \Cc^\infty(X_r, [0,\infty))\) satisfy \labelcref{eq:potential_on_end,eq:riccati_nonneg,eq:boundary_nonneg}.
        Then there exists a further constant \(C = C(X_0,g,k,\chi) > 0\), depending only on the data on \(X_0\), such that the following holds:
        
        For each asymptotically constant section \(u_r\) of \(S\) over \(X_r\) such that \(\modDirac^{\psi_r} u_r =0\) and
        \[
          \lvert u_r \rvert_{\mathcal{E}_\infty} = 1, \qquad 
          \langle u_r, \clm(\admMomentum_{\mathcal{E}}) \epsilon_0 u_r \rangle_{\mathcal{E}_\infty} = - \lvert\admMomentum_{\mathcal{E}}\rvert,
        \]
        we have the following estimate:
        \begin{equation}
            \frac{1}{2}(n-1)\omega_{n-1} \left( \admEnergy_{\mathcal{E}} - \lvert \admMomentum_{\mathcal{E}} \rvert \right) \geq \|\modnabla^{\psi_r} u_r\|_{\Lp^2(X_r)}^2 - \frac{C}{r} \|u_r\|^2_{\Lp^2_{-q}(K)}. \label{eq:applied_energy_momentum_bound}
        \end{equation}
    \end{lem}
    \begin{proof}
        We use \cref{lem:energy_momenum_estimate} and estimate
        \begin{align*}
        \frac{1}{2}(n-1)\omega_{n-1} \left( \admEnergy_{\mathcal{E}} - \lvert \admMomentum_{\mathcal{E}} \rvert \right) &\geq \|\modnabla^{\psi_r} u_r\|_{\Lp^2(X_r)}^2 + \int_{X_r} \zeta^{\psi_r}  \lvert u_r \rvert^2 \dV + \int_{\partial X_r} \eta^{\psi_r} \lvert u_r \rvert^2\dS, \\
        &\geq \|\modnabla^{\psi_r} u_r\|_{\Lp^2(X_r)}^2 - \int_{K} \zeta^{\psi_r}_- |u_r|^2 \dV, \quad \text{(using \labelcref{eq:applied_zeta,eq:boundary_nonneg})}\\
        & \geq \|\modnabla^{\psi_r} u_r\|_{\Lp^2(X_r)}^2 - \| \zeta^{\psi_r}_- \|_{\Lp^\infty_{-2}(K)} \|u_r\|^2_{\Lp^2_{-q}(K)}, \quad  \\
        &\geq \|\modnabla^{\psi_r} u_r\|_{\Lp^2(X_r)}^2 - \frac{1}{r} \|\D{\chi}\|_{\Lp^\infty_{-2}(K)}\ \|u_r\|^2_{\Lp^2_{-q}(K)}. \quad \text{(using \labelcref{eq:applied_zeta})}
        \end{align*}
        Thus we can choose \(C \coloneqq \|\D{\chi}\|_{\Lp^\infty_{-2}(K)}\).
    \end{proof}

    \begin{proof}[Proof of \cref{thm:spacetime_pmt_localized}]
        Let \((\mathcal{E},g,k)\) be an \(n\)-dimensional asymptotically flat initial data end.
        We fix a connected codimension zero submanifold \(X_0 \subseteq \mathcal{E}\) and a smooth cut-off function \(\chi \colon X_0 \to [0,1]\) as in \cref{setup:sweeping_out_M}~\labelcref{item:setup_X0}.
        In addition, we fix a section \(u_{00} \in \Ct^\infty(\mathcal{E},S)\) with \(\supp(u_{00}) \subset X_0 \setminus K\) which is asymptotically constant in \(\mathcal{E}\) and such that
        \[
          \lvert u_{00} \rvert_{\mathcal{E}_\infty} = 1, \qquad 
          \langle u_{00}, \clm(\admMomentum_{\mathcal{E}}) \epsilon_0 u_{00} \rangle_{\mathcal{E}_\infty} = - \lvert\admMomentum_{\mathcal{E}}\rvert.
        \]
        We now assume that \cref{setup:sweeping_out_M}~\labelcref{item:setup_Xr} can be realized for all \(r > 0\) and will prove that \(\admEnergy_{\mathcal{E}} \geq \lvert\admMomentum_{\mathcal{E}}\lvert\); this is precisely the contrapositive statement of \cref{thm:spacetime_pmt_localized}.
        
        For each \(r > 0\), choose \(\psi_r \colon X_r \to [0, \infty)\) as in \cref{lem:potential_exists}.
        Let \(c = c(X_0,g,k)\) and \(r_0 = r_0(X_0,g,k,\chi) > 0\) be the constants obtained from \cref{lem:isomorphism_in_setup}.
        Now for each \(r > r_0\), choose a section \(v_r \in \SobolevH^1_{-q}(X_r, S; \chi_1)\) such that
        \begin{equation}
          \modDirac^{\psi_r} v_r = - \modDirac(u_{00}) = - \modDirac^{\psi_r}(u_{00}),
          \label{eq:solution_condition}
        \end{equation}
        where the last equality holds because \(\psi_r\) vanishes on the support of \(u_{00}\) by construction.%
        Then \(u_r \coloneqq u_{00} + v_r\) is an asymptotically constant section which is asymptotic to \(u_{00}\) in \(\mathcal{E}\) and satisfies \(\modDirac^{\psi_r} u_r = 0\).
        Then we may estimate
        \begin{alignat*}{2}
        \frac{1}{2}(n-1)\omega_{n-1} \left( \admEnergy_{\mathcal{E}} - \lvert \admMomentum_{\mathcal{E}} \rvert \right) &\geq \|\modnabla^{\psi_r} u_r\|_{\Lp^2(X_r)}^2 - \frac{C}{r} \|u_r\|^2_{\Lp^2_{-q}(K)} \quad &&\text{(using \labelcref{eq:applied_energy_momentum_bound})} \\
        &\geq \|\modnabla^{\psi_r} u_r\|_{\Lp^2(X_r)}^2 - \frac{C}{r} \|v_r\|^2_{\Lp^2_{-q}(X_0)} \quad &&\text{(\(u_r = v_r\) on \(K\))}\\
        &\geq \|\modnabla^{\psi_r} u_r\|_{\Lp^2(X_r)}^2 - \frac{C}{r c} \|\modDirac^{\psi_r} v_r\|^2_{\Lp^2_{-q}(X_0)} \quad &&\text{(using \labelcref{eq:applied_partial_injectivity_bound})}\\
        &= \|\modnabla^{\psi_r} u_r\|_{\Lp^2(X_r)}^2 - \frac{C}{r c} \|\modDirac u_{00}\|^2_{\Lp^2_{-q}(X_0)}. \quad &&\text{(using \labelcref{eq:solution_condition})}
        \end{alignat*}
        In sum, we obtain
        \begin{equation}
            \frac{1}{2}(n-1)\omega_{n-1} \left( \admEnergy_{\mathcal{E}} - \lvert \admMomentum_{\mathcal{E}} \rvert \right) \geq \|\modnabla^{\psi_r} u_r\|_{\Lp^2(X_r)}^2 - \frac{C'}{r} \geq  - \frac{C'}{r},
            \label{eq:final_energy_momentum_estimate}
        \end{equation}
        where \(C' \coloneqq \frac{C}{c} \|\modDirac u_{00}\|^2_{\Lp^2_{-q}(X_0)}\) is a constant that only depends on objects chosen in advance on \(X_0\) independently of \(r\).
        Thus letting \(r \to \infty\) in \labelcref{eq:final_energy_momentum_estimate} proves that \(\admEnergy_{\mathcal{E}} - \lvert \admMomentum_{\mathcal{E}} \rvert \geq 0\), as desired.
    \end{proof}
    
    \begin{prop}\label{prop:parallel_spinors_exist}
        Let \((M,g,k)\) be a complete connected \(n\)-dimensional spin initial data set which contains a distinguished asymptotically flat end \(\mathcal{E}\).
        \begin{myenumi}
            \item \label{item:parallel_spinor_E_equals_P}If \(\admEnergy_{\mathcal{E}} = \lvert\admMomentum_{\mathcal{E}}\rvert\), then there exists a section \(u\) of \(S\) that satisfies \(\modnabla u = 0\) and is asymptotically constant in \(\mathcal{E}\) satisfying
            \[
              \lvert u \rvert_{\mathcal{E}_\infty} = 1, \qquad 
              \langle u, \clm(\admMomentum_{\mathcal{E}}) \epsilon_0 u \rangle_{\mathcal{E}_\infty} = - \lvert\admMomentum_{\mathcal{E}}\rvert = -\admEnergy_{\mathcal{E}}.
            \]
            \item \label{item:parallel_spinor_E_equals_0}If \(\admEnergy_{\mathcal{E}} = 0\), then for any asymptotically constant section \(u_{00}\) on \(\mathcal{E}\), one can find a section \(u\) with \(\modnabla u = 0\) that is asymptotic to \(u_{00}\).
        \end{myenumi}
    \end{prop}
    \begin{proof}
        \labelcref{item:parallel_spinor_E_equals_P} As in the proof of \cref{thm:spacetime_pmt_localized} above, we fix a connected codimension zero submanifold \(X_0 \subseteq \mathcal{E}\) and a smooth cut-off function \(\chi \colon X_0 \to [0,1]\) as in \cref{setup:sweeping_out_M}~\labelcref{item:setup_X0}.
        In addition, we fix a section \(u_{00} \in \Ct^\infty(X_0,S)\) which is asymptotically constant in \(\mathcal{E}\) and such that \(\lvert u_{00} \rvert_{\mathcal{E}_\infty} = 1\), \(\langle u_{00}, \clm(\admMomentum_{\mathcal{E}}) \epsilon_0 u_{00} \rangle_{\mathcal{E}_\infty} = - \lvert\admMomentum_{\mathcal{E}}\rvert = -\admEnergy_{\mathcal{E}}\).
        Since \(M\) is complete, spin and connected, we can find for each \(r > r_0\), a connected codimension zero submanifold \(X_0 \subset X_r \subseteq M\) which realizes \cref{setup:sweeping_out_M}~\labelcref{item:setup_Xr}.
        At this point the same argument as in the proof of \cref{thm:spacetime_pmt_localized} above applies and we obtain for each \(r > 0\) potentials \(\psi_r\) satisfying \labelcref{eq:potential_on_end,eq:riccati_nonneg,eq:boundary_nonneg,eq:potential_upper_bound} and sections \(u_r = u_{00} + v_r \in \Ct^\infty(X_r, S; \chi_1)\) (with \(v_r \in \SobolevH^{1}_{-q}(X_r, S; \chi_1)\)) that satisfy \(\modDirac^{\psi_r} u_r = 0\) and \labelcref{eq:final_energy_momentum_estimate}.
        But since \(\admEnergy_{\mathcal{E}} = \lvert\admMomentum_{\mathcal{E}}\rvert\), the estimate \labelcref{eq:final_energy_momentum_estimate} just says that
        \begin{equation}
            \|\modnabla^{\psi_r} u_r\|_{\Lp^2(X_r)}^2 \leq \frac{C'}{r}, \label{eq:connection_with_potential_to_zero}
        \end{equation}
        where \(C'\) is a constant independent of \(r\).
        Now fix an arbitrary connected codimension zero submanifold \(X_0 \subset X \subseteq M\).
        Then there exists an \(r_1 > r_0\) such that for each \(r > r_1\), we have \(X \subseteq X_r\) and \(\|\psi_r\|_{\Lp^\infty_{-1}(X)} \leq C_X/r\), where \(C_X\) is a constant depending on the data on \(X\).
        Here we have used the fact that \(\psi_r\) is compactly supported on \(X\) and \labelcref{eq:potential_upper_bound} to establish the latter estimate.
        Thus \cref{prop:partial_injectivity} implies that there exists an \(r_2 > r_1\) such that there exists a constant \(c_X > 0\), depending on the data on \(X\), such that
        \begin{equation}
            \forall r > r_2 \colon \quad \forall v \in \SobolevH^{1}_{-q}(X_r, S; \chi_1)\colon \quad \| \modDirac^{\psi_r} v \|_{\Lp^2(X_r)} \geq c_X \| v \|_{\Lp^2_{-q}(X)} \label{eq:partial_injectivity_once_again}
        \end{equation} 
        Thus we obtain for \(r,s > r_2\):%
        \begin{align*}
            \|\modnabla v_r - \modnabla v_s\|_{\Lp^2(X)} &= \|\modnabla u_r - \modnabla u_s\|_{\Lp^2(X)} \\
            &\leq \|\modnabla^{\psi_r} u_r - \modnabla^{\psi_s} u_s \|_{\Lp^2(X)} + \frac{1}{n} \left(\| \psi_r v_r\|_{\Lp^2(X)} + \|\psi_s v_s\|_{\Lp^2(X)}\right) \\
            &\leq \|\modnabla^{\psi_r} u_r\|_{\Lp^2(X)} + \|\modnabla^{\psi_s} u_s \|_{\Lp^2(X)} + \frac{C_X}{n} \left( \frac{1}{r} \|v_r\|_{\Lp^2_{-q}(X)} + \frac{1}{s} \|v_s\|_{\Lp^2_{-q}(X)}\right) \\
            &\leq \|\modnabla^{\psi_r} u_r\|_{\Lp^2(X)} + \|\modnabla^{\psi_s} u_s \|_{\Lp^2(X)} + \frac{2 C_X}{c_X n}  \| \modDirac u_{00}\|_{\Lp^2_{-q}(X_0)} \left(\frac{1}{r} + \frac{1}{s}\right),
        \end{align*}
        where in the last inequality we have used \labelcref{eq:partial_injectivity_once_again} and \(\modDirac^{\psi_r} v_r = -\modDirac(u_{00})\).
        Together with \labelcref{eq:connection_with_potential_to_zero} this proves that \(\|\modnabla v_r - \modnabla v_s\|_{\Lp^2(X)} \to 0\) as \(r,s \to \infty\).
        Finally, using the weighted Poincaré inequality for \(\modnabla\) on \(X\), this shows that \(v \coloneqq \lim_{r \to \infty} v_r\) exists in \(\SobolevH^1_{-q}(X,S)\) and \(u \coloneqq u_{00} + v\) satisfies \(\modnabla u = 0\).
        Since \(X\) was arbitrary, this already shows the existence of the desired section \(u\) on \emph{all of \(M\)}. This finishes the proof of \labelcref{item:parallel_spinor_E_equals_P}.
        
        Moreover, if \(\admEnergy_{\mathcal{E}} = 0\), then \(\admMomentum_{\mathcal{E}} = 0\) by \cref{neater.than.bartnikchrusciel} and so the same argument as above can be run starting with every asymptotically constant section \(u_{00}\) supported near the end \(\mathcal{E}\) because in this case the condition \(\langle u_{00}, \clm(\admMomentum_{\mathcal{E}}) \epsilon_0 u_{00} \rangle_{\mathcal{E}_\infty} = - \lvert\admMomentum_{\mathcal{E}}\rvert\) is always trivially satisfied.
        Thus the statement \labelcref{item:parallel_spinor_E_equals_0} also holds.
    \end{proof}

\section{Proof of \cref{rigidity.arb.ends}}\label{proof.of.rigidity}
The proof of \cref{rigidity.arb.ends} is essentially based on fundamental ideas from \citeauthor{BC96}~\cite{BC96} and \citeauthor{CM05}~\cite{MR2208148}.
The key ingredient is the following construction, the so-called \emph{Killing development}, which enables to find an ambient Lorentzian manifold that we can identify as Minkowski space.
Given a Riemannian manifold \((M,g)\) and a pair \((N,Y)\), where \(N \colon M \to \R\) is a smooth function and \(Y \in \Omega^1(M)\) a \(1\)-form, we may define
\begin{equation}
      \bar{M} \coloneqq \R \times M, \quad \bar{g} \coloneqq \proj_2^\ast(|Y|^2_g - N^2)\ \D{x_0}^2 + \proj_2^\ast Y \otimes \D{x_0} + \D{x_0} \otimes \proj_2^\ast Y + \proj_2^\ast g. \label{eq:Killing_development}
\end{equation}
We first observe that under a suitable uniformity assumption and completeness of the Riemannian manifold \(M\), the Killing development is a geodesically complete Lorentzian manifold.
\begin{lem}\label{lem:geodesically_complete}
    Let \((N,Y)\) be a scalar, \(1\)-form pair on a complete Riemannian manifold \((M,g)\) such that \(N^2 - |Y|^2_g \geq \delta\) for some \(\delta > 0\).%
    Then the \emph{Killing development}~\labelcref{eq:Killing_development} is a geodesically complete Lorentzian manifold.
\end{lem}
\begin{proof}
    We first show that \(\bar{g}\) is a Lorentzian metric.
    To this end, let \(\nu \coloneqq \partial_{x_0} - Y^\sharp\), where \(Y^\sharp\) is the vector field on \(M\) satisfying \(g(Y^\sharp, \xi) = Y(\xi)\) for all vector fields \(\xi\) on \(M\).
    Then \(\bar{g}(\nu, \xi) = 0\) for all vector fields \(\xi\) on \(M\) and \(\bar{g}(\nu,\nu) = \bar{g}(\partial_{x_0}, \partial_{x_0} - Y^\sharp) = |Y|^2 - N^2 - |Y|^2 = -N^2 \leq -\delta < 0\).
    Since \(\bar{g}\) restricts to the Riemannian metric \(g\) on \(M\), this shows that \(\bar{g}\) is Lorentzian.
    Moreover, by construction, translation along the \(\R\)-direction is an isometry of \(\bar{g}\).
    Thus the vector field \(\partial_{x_0}\) is a Killing vector field.
    
    To prove geodesic completeness, it suffices to show that, for every \(\bar{g}\)-geodesic \(c \colon [0,r) \to \mathcal{M}\) defined on a bounded half-open interval, there exists a compact subset \(K \subseteq \T\mathcal{M}\) of the tangent bundle such that \(c'(t) \in K\) for all \(t \in [0,r)\).
    Because then, by standard ODE theory, the geodesic \(c\) can be extended beyond the interval \([0,r)\) and so there exists no geodesic which stops in finite time.
    
    To this end, if we are given such a geodesic \(c \colon [0,r) \to \mathcal{M}\), we decompose it as \(c(t) = (\alpha(t), \gamma(t))\) with \(\alpha(t) \in \R\) and \(\gamma(t) \in M\).
    Then we can write \(c'(t) = \alpha'(t) \oplus \gamma'(t)\) with \(\alpha'(t) \in  \T_{\alpha(t)} \R =\R\) and \(\gamma'(t) \in \T_{\gamma(t)} M\).
    Since \(c\) is a geodesic, there exists a constant \(k_1 \in \R\) such that
    \begin{equation}
        k_1 = \bar{g}(c'(t), c'(t)) = (|Y|^2_g - N^2) \alpha'(t)^2 + 2 \alpha'(t) Y(\gamma'(t)) + g(\gamma'(t), \gamma'(t)) \label{eq:geodesic_constant_speed}
    \end{equation}
    for all \(t \in [0, r)\).
    Moreover, since \(\partial_{x_0}\) is a Killing field, there exists a constant \(k_2 \in \R\) such that
    \begin{equation}
        k_2 = \bar{g}(\partial_{x_0}, c'(t)) = (|Y|^2_g - N^2) \alpha'(t) + Y(\gamma'(t))
        \label{eq:killing_component_constant}
    \end{equation}
    for all \(t \in [0,r)\).
    Combining \labelcref{eq:geodesic_constant_speed,eq:killing_component_constant}, we derive that
    \begin{equation}
        0 \leq g(\gamma'(t), \gamma'(t)) =  k_1 - (N^2 - |Y|^2_g) \alpha'(t)^2 - 2 k_2 \alpha'(t) \leq k_1 - \delta \alpha'(t)^2 - 2 k_2 \alpha'(t)
        \label{eq:completness_main_estimate}
    \end{equation}
    for all \(t \in [0,r)\).
    Since the right-hand side tends to \(-\infty\) as \(|\alpha'(t)| \to \infty\), it follows that \(a = \sup_{t \in [0,r)} |\alpha'(t)| < \infty\).
    But then \(b = \sup_{t \in [0,r)} |\gamma'(t)|_g \leq \sqrt{|k_1| + 2 a |k_2|} < \infty\).
    It follows that \(\alpha(t) \in K_1 \coloneqq [\alpha(0) - ar, \alpha(0) + ar]\) and \(\gamma(t) \in K_2 \coloneqq \overline{\Ball}_{br}(\gamma(0))\) for all \(t \in [0,r)\). 
    Since \((M,g)\) is a complete Riemannian manifold, the set \(K_2\) is compact.
    Thus we conclude that \(c'(t) \in K\) for all \(t \in [0,r)\), where \(K \subset \T \mathcal{M}\) is the compact subset defined as
    \[
        K \coloneqq \{ (\eta_s, \xi_x) \in \T\R \times \T M = \T \mathcal{M} \mid s \in K_1, x \in K_2, |\eta_s| \leq a, |\xi_x|_g \leq b\}.
        \qedhere
    \]
\end{proof}

Given a suitable section \(u\) of a spinor bundle \(S \to M\), where \((M,g)\) is a spin manifold and \(S\) is as in \cref{sec:diracwitten_potential}, we will define
\begin{equation}\label{defining.pair}
    N= \langle u,u \rangle\:\:\textnormal{and}\:\: Y(\xi)=\langle \clm(\xi^\flat) \epsilon_0 u,u \rangle
\end{equation}
and consider the corresponding Killing development~\labelcref{eq:Killing_development}.
It turns out that if the spinor \(u\) is \(\modnabla\)-parallel, then \(u\) extends to a parallel spinor on the spacetime \((\bar{M} = \R \times M, \bar{g})\) and the initial data set \((M,g,k)\) embeds into \((\bar{M}, \bar{g})\) as \(\{x_0\} \times M\) for arbitrary \(x_0 \in \R\) (in particular, this means that the second fundamental form of \(\{x_0\} \times M\) inside \(\bar{M}\) with respect to \(\bar{g}\) agrees with \(k\)); see \cite[Theorem~4.5]{ammann2021dominant}.

\begin{lem}
    Suppose that \(\modnabla u = 0\) and let \((N,Y)\) be as in \labelcref{defining.pair}.
    Then \(\D (N^2 - |Y|^2) = 0\).\label{lem:vanishing_derivative}
\end{lem}
\begin{proof}
    Note that \(\modnabla u = 0\) means that \(\nabla_\xi u = - \frac{1}{2} \clm(k(\xi,\blank)) \epsilon_0 u\) for all vector fields \(\xi\) by definition of \(\modnabla\).
    Then%
    \begin{align*}
        (\nabla_\xi Y)(\eta) &= \nabla_\xi (Y(\eta)) - Y(\nabla_\xi \eta) \\
        &= \nabla_\xi \langle \clm(\eta^\flat) \epsilon_0 u, u \rangle - \langle \clm(\nabla_\xi \eta^\flat) \epsilon_0 u, u \rangle \\
        &= -\frac{1}{2} \left( \langle \clm(\eta^\flat) \epsilon_0 \clm(k(\xi, \blank)) \epsilon_0 u, u\rangle + \langle \clm(\eta^\flat) \epsilon_0 u, \clm(k(\xi, \blank)) \epsilon_0 u \rangle\right) \\
        &= \frac{1}{2} \langle (\clm(\eta^\flat) \clm(k(\xi, \blank)) + \clm(k(\xi, \blank)) \clm(\eta^\flat) )u, u\rangle \\
        &= - k(\xi, \eta) \langle u, u \rangle.
    \end{align*}
    Thus \(\nabla_\xi Y = -k(\xi, \blank) N\) and so
    \begin{align*}
        \nabla_\xi (N^2 - |Y|^2) &= 4 N \langle \nabla_\xi u, u \rangle - 2 \langle \nabla_\xi Y, Y \rangle \\
        &= -2N \langle \clm(k(\xi, \blank)) \epsilon_0 u, u \rangle + 2 N \langle k(\xi, \blank), Y\rangle \\
        &= -2N \langle \clm(k(\xi, \blank)) \epsilon_0 u, u \rangle + 2 N \langle \clm(k(\xi, \blank))\epsilon_0 u, u \rangle = 0. \qedhere
    \end{align*}
\end{proof}

We may now start the proof of \labelcref{rigidity.arb.end:item1} in \cref{rigidity.arb.ends}, that is, the implication that $E=0$ implies the existence of an embedding of $(M,g,k)$ into Minkowski space.

\begin{proof}[Proof of \cref{rigidity.arb.ends}~\labelcref{rigidity.arb.end:item1}]
    Let \(u_{00} \in \Ct^\infty(\mathcal{E}, S)\) be an asymptotically constant section with
    \begin{equation}
        \lvert u_{00} \rvert_{\mathcal{E}_\infty} = 1, \qquad \forall \xi \in \R^n \colon \langle \clm(\xi^\flat) \epsilon_0 u_{00}, u_{00} \rangle_{\mathcal{E}_\infty} = 0. \label{eq:rigidity_spinor_asymptotic_conditions}
    \end{equation}
    By \cref{prop:parallel_spinors_exist}~\labelcref{item:parallel_spinor_E_equals_0}, there exists a $\modnabla$-parallel spinor $u$ asymptotic to $u_{00}$ on \(\mathcal{E}\) in the sense that $u - u_{00} \in \SobolevH^1_{-q}$ on the end $\mathcal{E}$. 
    We then define \((N,Y)\) by \labelcref{defining.pair} and consider the corresponding Killing development \labelcref{eq:Killing_development}.
    Then \labelcref{eq:rigidity_spinor_asymptotic_conditions} implies that \(N \to 1\) and \(Y \to 0\) as we go to infinity in the end \(\mathcal{E}\).
    By \cref{lem:vanishing_derivative}, this means that \(N^2 - |Y|^2 = 1\) everywhere and \cref{lem:geodesically_complete} implies that the Killing development \(\bar{M}\) is a geodesically complete Lorentzian manifold.
    
    Furthermore, we may directly apply \cite[Theorem~4.5]{ammann2021dominant} to see that the Lorentzian manifold \(\bar{M}\) contains the initial data set \((M,g,k)\) as \(\{x_0\} \times M\) for each \(x_0 \in \R\) and that the Killing vector field \(\partial_{x_0}\) is parallel on \(\bar{M}\).
    
    Next we will prove that \(\bar{M}\) is already flat.
    Indeed, by \cref{prop:parallel_spinors_exist}~\labelcref{item:parallel_spinor_E_equals_0}, we actually obtain that there exists a family \((u_{\alpha})\) of \(\modnabla\)-parallel sections which forms a global frame of the bundle \(S\).
    Then, as is described in the proof of \cite[Theorem~4.5]{ammann2021dominant}, the spinor bundle \(S\) can be canonically extended to a spinor bundle \(\bar{S} \to \bar{M}\) corresponding to the Lorentzian metric \(\bar{g}\).
    Moreover, restriction of sections to \(M \times \{0\}\) yields an isomorphism,
    \[
      \{\bar{u} \in \Ct^\infty(\bar{M}, \bar{S}) \mid \bar{\nabla}_{\partial_{x_0}} \bar{u} = 0 \} \xrightarrow{\cong} \Ct^\infty(M, S).
    \]
    In particular, we find unique extensions \(\bar{u}_\alpha \in \Ct^\infty(\bar{M}, \bar{S})\) of \(u_{\alpha} \in \Ct^\infty(M, S)\) such that \(\bar{\nabla}_{\partial_{x_0}} \bar{u}_\alpha = 0\).
    By construction, the Lorentzian spinor connection on \(\bar{S}\) restricts to \(\modnabla\) on \(S\), and hence \((\bar{u}_\alpha)_\alpha\) actually form a basis of parallel sections of \(\bar{S}\).
    This means that \(\bar{S}\) is flat and so the underlying Lorentzian manifold \((\bar{M}, \bar{g})\) must already be flat.
    
    Then a standard classification result~\cite[Theorem~2.4.9]{Wolf} implies that the universal cover of \((\bar{M}, \bar{g})\) is isometric to Minkowski space.
    Finally, we claim that \(\bar{M}\) is simply-connected and hence \(\bar{M}\) is itself isometric to Minkowski space.
    To show that \(\bar{M} = \R \times M\) is simply-connected, it is enough to show that \(M\) itself is simply-connected.
    Indeed, assume by contradiction that \(M\) is not simply-connected.
    Then the universal covering of \(M\) would have more than one end since it would contain multiple disjoint copies of the end \(\mathcal{E}\) because \(\mathcal{E}\) is a simply-connected open subset of \(M\).
    However, the universal covering of \(M\) is contractible because it is homotopy equivalent to the universal covering of \(\bar{M} = \R \times M\) which is diffeomorphic to \(\R^{n+1}\).
    But a contractible manifold of dimension \(> 1\) has necessarily only one end,\footnote{This is an elementary consequence of Poincaré duality: If a contractible manifold \(X\) had more than one end, then there exists an increasing exhaustion \((K_i)_i\) by compact subsets such that \(\varinjlim_i \tilde{\mathrm{H}}^0(X \setminus K_i) \neq 0\). 
    But then \(0 \neq \varinjlim_i \tilde{{\mathrm{H}}}^0(X \setminus K_i) \cong \varinjlim_i \mathrm{H}^1(X, X \setminus K_i) = \mathrm{H}^1_{\mathrm{c}}(M) \overset{\mathrm{PD}}{\cong} \mathrm{H}_{n-1}(M) = 0\), a contradiction, where contractibility is used in the first isomorphism and in the final equality (together with \(n > 1\)), and Poincaré duality between homology and compactly supported cohomology is used in the last isomorphism.} so this yields a contradiction.%
    
    Thus $(M,g,k)$ embeds as desired in Minkowski space \((\bar{M},\bar{g})\).
\end{proof}

We now turn to part~\labelcref{rigidity.arb.end:item2} of \cref{rigidity.arb.ends} ($E_\mathcal{E}=|P_\mathcal{E}|\Rightarrow E_\mathcal{E}=0$), which is essentially a repeat of \cite{CM05}.
The key part is the following computational lemma we take from \cite{CM05}. 
\begin{lem}[{\cite[Theorem~2.5]{CM05}}]\label{rigidity.comp}
Let $R>0$ and $(g,k)$ be initial data on $\mathbb{R}^n\backslash \Disk_R(0)$ satisfying 
\begin{equation}
    g_{ij}-\delta_{ij}=\bigO_{3+\lambda}(|x|^{-\alpha})\:\:,\:\: k_{ij}=\bigO_{2+\lambda}(|x|^{-1-\alpha})
\end{equation}
\begin{equation}
    \alpha>\  \begin{cases}
    \quad \frac{1}{2}&  \quad n=3\\
   \quad n-3  & \quad n\geq 4
\end{cases} \quad ,\quad \epsilon>0\quad,\quad 0<\lambda<1
\end{equation}
\begin{equation}\label{decay.on.mu}
J^i=\bigO_{1+\lambda}(|x|^{-n-\epsilon})\:\:,\:\: \mu=\bigO_{1+\lambda}(|x|^{-n-\epsilon})    
\end{equation}
Let $N$ be a scalar field and $Y^i$ a vector field on $\mathbb{R}^n\backslash B(R)$ such that 
\begin{equation*}
    N\to_{r\to\infty} A^0\quad,\quad Y^i\to_{r\to\infty} A^i\quad,\quad (A^0)^2=|A|^2
\end{equation*}
for some constants $A^\mu\neq 0$. Suppose moreover that 
\begin{equation}\label{killing.condition}
2Nk_{ij}+\mathcal{L}_Yg_{ij}=\bigO_{3+\lambda}(|x|^{-(n-1)-\epsilon})
\end{equation}
\begin{equation}\label{decay.on.tau}
    \tau_{ij}=\bigO_{1+\lambda}(|x|^{-n-\epsilon}).
\end{equation}
Then $(\admEnergy,\admMomentum)=(0,0)$. 
\end{lem}
\begin{rem}
$\mathcal{L}$ denotes the Lie derivative and $\tau_{ij}$ is defined as in Proposition 2.1 of \cite{CM05}
\begin{equation}
\tau_{ij}-\frac{1}{2}g^{kl}\tau_{kl}g_{ij}=\Ric_{ij}+g^{ml}k_{lm}k_{ij} -2k_{ik}k_{lj}g^{lk} - N^{-1}(\mathcal{L}_Yk_{ij}+\nabla_i\nabla_j N)-\frac{\mu}{2}g_{ij}.
\end{equation}
\end{rem}
\begin{rem}
The proof of Lemma \ref{rigidity.comp}, which is contained in \cite{BC96,CM05}, solely involves asymptotic analysis on $\mathcal{E}$ and applies unchanged in our setting, where possibly other arbitrary but unrelated ends are present.
\end{rem}

\begin{proof}[Proof of \cref{rigidity.arb.ends}~\labelcref{rigidity.arb.end:item2}]
Assume by contradiction that \(\admMomentum_{\mathcal{E}} \neq 0\).
By \cref{prop:parallel_spinors_exist}, we have a section $u$ of \(S\) satisfying 
\begin{equation}
    |u|_{\mathrm{E}_\infty} = 1, \qquad \langle u, \clm(\admMomentum_{\mathcal{E}}) \epsilon_0 u \rangle_{\mathcal{E}_\infty} = - \lvert\admMomentum_{\mathcal{E}}\rvert = -\admEnergy_{\mathcal{E}},\label{eq:rigidity_spinor_extra_decay_asymptotic_conditions}
\end{equation}
and $\modnabla u=0$.
Again, we use it to define a scalar, $1$-form pair $(N,Y)$ as in \eqref{defining.pair}.
By \labelcref{eq:rigidity_spinor_extra_decay_asymptotic_conditions} and the fact that we have assumed \(\admMomentum_{\mathcal{E}} \neq 0\), we obtain that both \(N \to 1\) and \(|Y|^2 \to 1\) as we go to infinity in \(\mathcal{E}\).
Using \cref{lem:vanishing_derivative}, we thus obtain \(N^2 = |Y|^2\).
 
Then, by the computation from line (3.18) to (3.28) in \cite{CM05}, we arrive at
\begin{equation}\label{asymp.approach}
N^2\tau_{ij}= \mu Y_i Y_j 
\end{equation}
 where $\tau_{ij}$ is as above.
 This can be taken verbatim from \cite{CM05} because it only involves $\modnabla$. \eqref{asymp.approach} and the decay assumption on $\mu$ in Definition \ref{extra.decay} imply that $\tau_{ij}$ satisfies the decay assumption in \eqref{decay.on.tau}.

Moreover, the computation from line (3.15) to (3.27) in \cite{CM05} shows that the assumption on $J^i$ in \eqref{decay.on.mu} is satisfied. Finally, the computation between line (3.27) and (3.28) in \cite{CM05} shows that \eqref{killing.condition} in Lemma \ref{rigidity.comp} is satisfied. Therefore, we can apply Lemma \ref{rigidity.comp} to obtain that $(\admEnergy_\mathcal{E},\admMomentum_\mathcal{E})=0$, a contradiction.
\end{proof}

\appendix

\section{Weighted Poincaré inequalities in the presence of boundary}\label{weightedpoincare}
In this appendix we discuss weighted Poincaré inequalities adapted to our setting, that is, on asymptotically flat manifolds with compact boundary for sections of Hermitian vector bundles with respect to not-necessarily metric connections.
This is essentially folklore~(compare\ \cite{BC05,Bartnik:MassAsymptoticallyFlat}).

\begin{lem}\label{lem:weighted_poincare_at_infinity}
    Let \((\R^n \setminus \Disk_{d}(0),g)\) be an asymptotically flat end of some decay rate \(\tau > (n-2)/2\).
    Then for each \(\delta < 0\), there exists a constant \(C = C(g,d, \delta) > 0\) and \(r_0 = r_0(g,d,\delta) > d\) such that the following holds: For each \(r \geq r_0\) and \(u \in \SobolevH^1_{\delta}(\R^n \setminus \Ball_{r}(0))\), we have
    \[
        \| u \|_{\Lp^2_{\delta}(\R^n \setminus \Ball_r(0))} \leq C \| |\nabla u|_g \|_{\Lp^2_{\delta-1}(\R^n \setminus \Ball_r(0))},
        \]
    where the weighted \(\Lp^2\)-norms are calculated with respect to the weight function \(\rho = |x|\) and the volume measure induced by the Riemannian metric \(g\).
\end{lem}
\begin{proof}
    It is enough to consider compactly supported smooth functions \(u \in \Cc^\infty(\R^n \setminus \Ball_r(0)) \).
    We essentially follow the proof given by Lee in \cite[Proposition~A.28]{Lee2019-GeometricRelativity} (which, in turn, follows Bartnik~\cite{Bartnik:MassAsymptoticallyFlat}) except that we need to treat an additional boundary term in our setting.
    The first part of the proof of \cite[Proposition~A.28]{Lee2019-GeometricRelativity} is unaffected by the presence of a boundary and we obtain for each \(r > d\) the estimate
    \begin{align}
      &\int_{\R^{n} \setminus \Ball_r(0)} \langle \nabla^g(\rho^{2-n}), \nabla^g (\rho^{-2 \delta} |u|^2) \rangle \dV_g  \nonumber \\
      &\leq \int_{\R^n \setminus \Ball_r(0)} - 2 \delta (2-n) |\nabla^g \rho|^2 \rho^{-2\delta -n} |u|^2 \dV_g + C' \|u\|_{\Lp^2_{\delta}} \cdot \| \nabla^g u \|_{\Lp^2_{\delta-1}} \label{eq:weighted_poincare_estimate1},
    \end{align}
    where \(C'\) is a constant that only depends on the metric \(g\).
    Using partial integration, we also obtain the identity
    \begin{align*}
        &\int_{\R^{n} \setminus \Ball_r(0)} \langle \nabla^g(\rho^{2-n}), \nabla^g (\rho^{-2 \delta} |u|^2) \rangle \dV_g  \\
    &= \int_{\R^n \setminus \Ball_r(0)} \nabla^\ast \nabla^g(\rho^{2-n}) \rho^{-2\delta} |u|^2 \dV_g + \int_{\Sphere^{n-1}_r} \nabla^g_\nu(\rho^{2-n}) \rho^{-2 \delta} |u|^2 \dS,
    \end{align*}
    where \(\nu\) denotes outer unit normal with respect to the metric \(g\) along \(\partial (\R^n \setminus \Ball_r(0)) \), that is, pointing towards the deleted interior disk \(\Ball_r(0)\).
    Since \(2-n < 0\), we have \(\nabla^g_\nu (\rho^{2-n}) \geq 0\) along \(\Sphere_r^{n-1}\), and thus we deduce 
    \begin{equation}
        \int_{\R^{n} \setminus \Ball_r(0)} \langle \nabla^g(\rho^{2-n}), \nabla^g (\rho^{-2 \delta} |u|^2) \rangle \dV_g \geq \int_{\R^n \setminus \Ball_r(0)} \nabla^\ast \nabla^g(\rho^{2-n}) \rho^{-2\delta} |u|^2 \dV_g \label{eq:weighted_poincare_estimate2}
    \end{equation}
    Chaining together the estimates \labelcref{eq:weighted_poincare_estimate1,eq:weighted_poincare_estimate2} an rearranging terms, we arrive at the estimate
    
    \begin{equation}
        \int_{\R^n \setminus \Ball_r(0)} \left( \nabla^\ast \nabla^g(\rho^{2-n}) \rho^{-2\delta} |u|^2 + 2 \delta (2-n) |\nabla^g \rho|^2 \rho^{-2 \delta -n } |u|^2 \right) \dV_g
        \leq C' \|u\|_{\Lp^2_\delta} \cdot \| \nabla^g u \|_{\Lp^2_{\delta-1}}
        \label{eq:poincare_penultimate_step}
    \end{equation}
    Since \(\rho^{2-n}\) is harmonic with respect to the Euclidean background metric, we obtain that
    \[
        \nabla^\ast \nabla^g (\rho^{2-n}) = -\Delta_g(\rho^{2-n}) = \Delta_{g_{\mathrm{eukl}}}(\rho^{2-n}) - \Delta_g(\rho^{2-n}) \in \bigO(|x|^{-\tau-n}).
    \]
    In particular, if \(r \geq r_0\) for \(r_0 = r_0(g, d, \delta)\) sufficiently large, then we can ensure that
    \begin{equation}
        \nabla^\ast \nabla^g(\rho^{2-n}) \rho^{-2\delta} |u|^2 + 2 \delta (2-n) |\nabla^g \rho|^2 \rho^{-2 \delta -n } |u|^2  \geq C'' |u|^2 \rho^{-2 \delta -n}
        \label{eq:use_asymptoticaly_harmonic}
    \end{equation}
    on all of \(\R^n \setminus \Ball_r(0)\) for a constant \(C'' = C''(g, \delta, r_0) >0\).
    Here we have used that \(\delta < 0\).
    Then \labelcref{eq:poincare_penultimate_step,eq:use_asymptoticaly_harmonic} show that for all \(r \geq r_0\), we have
    \[
        C'' \|u\|_{\Lp^2_{\delta}(\R^n \setminus \Ball_r(0))}^2 \leq C' \|u\|_{\Lp^2_\delta(\R^n \setminus \Ball_r(0))} \cdot \| \nabla^g u \|_{\Lp^2_{\delta-1}(\R^n \setminus \Ball_r(0))}.
    \]
    This proves the desired estimate with \(C = C' / C''\).
\end{proof}

\begin{lem}\label{lem:parallel_vanish}
    Let \((X,g)\) be a complete connected asymptotically flat manifold with compact interior boundary \(\partial X\).
    Let \(E \to X\) be a Hermitian vector bundle endowed with a metric connection \(\nabla\).
    Let \(\delta < 0\) and \(A\) be a smooth \(1\)-form on \(X\) with values in the endomorphisms of \(E\) such that \(\lvert A \rvert \in \bigO(\rho^{\delta-1})\) on each asymptotically flat end of \(X\).
    Define a new connection \(\nabla^A = \nabla + A\) on \(E\).
    Then each \(u \in \SobolevH^1_{\delta}(X,E)\) which satisfies \(\nabla^A u = 0\) must vanish.
\end{lem}
\begin{proof}
    Assume to the contrary that \(u \neq 0\) and \(\nabla^A u = 0\).
    Then \(u\) is smooth and it must be non-zero at each point because \(X\) is connected and \(u\) satisfies a linear ODE along each smooth curve in \(X\).
    Now select an asymptotically flat end \(\R^n \setminus \Disk_d(0) \cong \mathcal{E} \subseteq X\) and choose \(r_0\) and \(C\) as in \cref{lem:weighted_poincare_at_infinity} above.
    Note that since \(\nabla\) is a metric connection \(E\), we have Kato's inequality \(\lvert \nabla \lvert u \rvert \rvert \leq \lvert \nabla u \rvert\).
    We thus find for each \(r \geq r_0\) that
    \begin{align*}
      \frac{1}{C} \|u\|_{\Lp^2_{\delta}(\R^n \setminus \Ball_r(0))} &\leq \| \nabla |u|\|_{\Lp^2_{\delta-1}(\R^n \setminus \Ball_r(0))} \\
      &\leq \| \nabla u\|_{\Lp^2_{\delta-1}(\R^n \setminus \Ball_r(0))} = \| A u \|_{\Lp^2_{\delta-1}(\R^n \setminus \Ball_r(0))} \leq \| A \|_{\Lp_{-1}^\infty(\R^n \setminus \Ball_r(0))}\ \| u \|_{\Lp^2_{\delta}(\R^n \setminus \Ball_r(0))}.
    \end{align*}
    Since \(u\) vanishes nowhere, this implies that \(0 < 1/C \leq \| A \|_{\Lp_{-1}^\infty(\R^n \setminus \Ball_r(0))}\).
    But since \(|A| \in \bigO(\rho^{\delta-1})\) and \(\delta - 1 < -1\), we have that \(\| A \|_{\Lp_{-1}^\infty(\R^n \setminus \Ball_r(0))} \to 0\) as \(r \to \infty\), a contradiction.
    \end{proof}

\begin{prop}[Weighted Poincaré inequality]\label{lem:weighted_poincare_modnabla}
    Let \((X,g)\) be a complete connected asymptotically flat manifold with compact interior boundary \(\partial X\).
    Let \(E \to X\) be a Hermitian vector bundle endowed with a metric connection \(\nabla\).
    Let \(\delta < 0\) and \(A\) be a smooth \(1\)-form on \(X\) with values in the endomorphisms of \(E\) such that \(\lvert A \rvert \in \bigO(\rho^{\delta-1})\) on each asymptotically flat end of \(X\).
    Then there exists a constant \(C = C(X,g,E,\delta,A)\) such that for all \(u \in \SobolevH^1_{\delta}(X,E)\) we have
    \[
        \| u \|_{\Lp^2_{\delta}(X)} \leq C \| \nabla^A u \|_{\Lp^2_{\delta - 1}(X)}.
    \]
\end{prop}
\begin{proof}
    The estimate 
    \begin{equation}
        \| \nabla u \|_{\Lp^2_{\delta-1}(X)} \leq \| \nabla^A u \|_{\Lp^2_{\delta-1}(X)} + \| A u \|_{\Lp^2_{\delta-1}(X)} \leq \| \nabla^A u \|_{\Lp^2_{\delta-1}(X)} + \|A\|_{\Lp^\infty_{\delta-1}(X)}\ \|u\|_{\Lp^2_{0}(X)}
        \label{eq:estimate_nabla_by_modnabla}
    \end{equation}
    holds for all \(u \in \SobolevH^1_{\delta}(X,E)\).
    Assume to the contrary that such a constant \(C > 0\) does not exist.
    Then there exists a sequence \(u_i \in \SobolevH^1_{\delta}(X,E)\) such that \(\|u_i\|_{\Lp^2_{\delta}(X)} = 1\) and \(\|\nabla^A u_i \|_{\Lp^2_{\delta-1}(X)} \to 0\).
    Then because of \labelcref{eq:estimate_nabla_by_modnabla} the sequence \((u_i)\) is uniformly bounded in \(\SobolevH^1_{\delta}(X,E)\). 
    By the weighted Rellich theorem\footnote{see e.g.~\cite[Lemma~2.1]{Choquet-Bruhat_Christodoulou:EllipticSystems} which immediately extends to our setting with boundary.}, we can pass to a subsequence and assume that \(u_i \to u\) in \(\Lp^2_0(X,E)\).
    But then \labelcref{eq:estimate_nabla_by_modnabla} implies that \(\|\nabla (u_i - u_j)\|_{\Lp^2_{\delta-1}(X)} \to 0\) as \(i,j \to \infty\).
    This, in turn, via \cref{lem:weighted_poincare_at_infinity} and using the fact that the sequence \((u_i)\) converges locally in \(\Lp^2\) (as a consequence of \(\Lp^2_0\)-convergence) implies that \(\|u_i - u_j\|_{\Lp^2_{\delta}(X)} \to 0\) as \(i,j \to \infty\).
    We thus conclude that the sequence \((u_i)\) converges in \(\SobolevH^1_{\delta}(X,E)\) and so \(u_i \to u\) in \(\SobolevH^1_{\delta}(X,E)\).
    Hence \(\|u\|_{\Lp^2_{\delta}(X)} = 1\) and \(\nabla^A u = 0\), a contradiction to \cref{lem:parallel_vanish}.
\end{proof}

\begin{rem}\label{rem:apply_weighted_poincare}
    In this paper, we use the weighted Poincaré inequality for the case \(\delta = -q = -(n-2)/2\), in which case it simplifies to \(\| u \|_{\Lp^2_{-q}} \leq C \| \nabla u \|_{\Lp^2}\).
    Moreover, we are primarily interested in applying it to the following setup:
    
    Let \((X,g,k)\) be a complete asymptotically flat initial data set, where \((X,g)\) is endowed with a spin structure, and let \(S \to X\) be the variant of the spinor bundle as introduced in \cref{sec:diracwitten_potential}.
    Then both the modified connection \(\modnabla_\xi = \nabla_\xi + 1/2 \clm(k(\xi, \blank)) \epsilon_0\) and \(\modnabla^\psi_\xi = \modnabla_\xi - \frac{\psi}{n} \clm(\xi) \epsilon_1\) with \(\psi \in \Cc^\infty(X)\) can be written in the form \(\nabla + A\)  for suitable \(1\)-forms \(A\) satisfying \(|A| \in \bigO(\rho^{-q-1})\).
    Thus we obtain weighted Poincaré inequalities for both of these modified connections from \cref{lem:weighted_poincare_modnabla}.
\end{rem}

\printbibliography

\end{document}